\documentclass[11pt,twoside]{amsart}
\usepackage{latexsym,amsfonts,amsmath,amssymb}

\newtheorem{Th}{Theorem}[section]
\newtheorem{Rem}[Th]{Remark}
\newtheorem{Ex}[Th]{Example}
\newtheorem{Lemma}[Th]{Lemma}
\newtheorem{Def}[Th]{Definition}
\newtheorem{Prop}[Th]{Proposition}
\newtheorem{Cor}[Th]{Corollary}

\catcode`\@=11

\renewcommand{\section}%
   {\setcounter{equation}{0}\@startsection {section}{1}{\z@}{-3.5ex plus -1ex
  minus -.2ex}{2.3ex plus .2ex}{\Large\bf}}

\def\supp{\mathop{\rm supp}\nolimits}

\def\singsupp{\mathop{\rm sing\, supp}\nolimits}

\def\Re{\mathop{\rm Re}\nolimits}
\def\Im{\mathop{\rm Im}\nolimits}

\def\grad{\mathop{\rm grad}\nolimits}

\def\ds{\displaystyle}
\def\R{\mathbb R}
\def\C{\mathbb C}
\def\N{\mathbb N}
\def\Q{\mathbb Q}
\def\WF{\mathop{\rm WF}\nolimits}
\def\proj{\mathop{\mbox{\rm proj}}}
\def\ind{\mathop{\mbox{\rm ind}}}

\newcommand{\D}{\mathcal{D}}
\newcommand{\E}{\mathcal{E}}
\newcommand{\F}{\mathcal{F}}

\newcommand{\afrac}[2]{\genfrac{}{}{0pt}{1}{#1}{#2}}

\newcommand{\beqsn}{\arraycolsep1.5pt\begin{eqnarray*}}
\newcommand{\eeqsn}{\end{eqnarray*}\arraycolsep5pt}
\newcommand{\beqs}{\arraycolsep1.5pt\begin{eqnarray}}
\newcommand{\eeqs}{\end{eqnarray}\arraycolsep5pt}

\title{A characterization of the wave front set defined by the iterates
of an operator with constant coefficients}

\author{Chiara Boiti}
\address{
Dipartimento di Matematica e Informatica \\Universit\`a di Ferrara\\
Via Ma\-chia\-vel\-li n.~30\\
44121 Ferrara\\
Italy}
\email{chiara.boiti@unife.it}

\author{David Jornet}
\address{
Instituto Universitario de Matem\'atica Pura y Aplicada IUMPA\\
Universitat Po\-li\-t\`ecni\-ca de Val\`encia\\
C/Camino de Vera, s/n\\
E-46071 Valencia\\
Spain}
\email{djornet@mat.upv.es}

\begin{document}

\begin{abstract}
We characterize the wave front set
$\WF^P_\ast(u)$ with respect to the iterates of a linear partial differential
operator with constant coefficients of a classical distribution
$u\in\D'(\Omega)$, $\Omega$ an open subset in $\R^n$. We use recent Paley-Wiener theorems
for generalized ultradifferentiable classes in the sense of Braun, Meise and
Taylor. We also give several examples and applications to the regularity of
operators with variable coefficients and constant strength. Finally,
we construct a distribution with prescribed wave front set of this type.
\end{abstract}

\maketitle
\markboth{\sc A characterization of the wave front set\ldots}
{\sc C.~Boiti and D.~Jornet}

\vspace{3mm}
\noindent {\em keywords\,}: {Iterates of an operator, wave front set,
ultradifferentiable functions.}

\noindent {\em 2010 Mathematics Subject Classification\,}:
{35A18, 35A20, 35A21.}

\section{Introduction}

We introduced in \cite{BJJ} the wave front set $\WF^{P}_{*}(u)$, for a classical distribution $u\in\D'(\Omega)$  defined on a open set $\Omega$ of $\R^{n}$, with respect to the
iterates of a hypoelliptic linear partial differential operator $P$
with constant coefficients for ultradifferentiable classes in the sense of
Braun, Meise and Taylor \cite{BMT}.  We established
in \cite{BJJ}
a microlocal regularty theorem for this wave front set and we studied the
product of ultradifferentiable functions defined in the usual way with the
ones defined by iterates. As a consequence we obtained a partial result
related to the construction of distributions with prescribed wave front sets.
Here, we describe more precisely the behaviour of the set $\WF^{P}_{*}(u)$ and
complete the previous result about prescribed singularities. Finally,
we give some applications to the $\omega$-micro-regularity of linear partial
differential operators with variable coefficients and constant strength.

The problem of iterates begins mainly when Komatsu~\cite{K1} in 1960s
characterized
analytic functions $f$ in
terms of the behaviour  of
successive iterates $P(D)^jf$ of a partial differential elliptic
operator $P(D)$ with constant coefficients,
proving that a $C^\infty$ function $f$ is real analytic
in $\Omega$ if and only if for every compact set $K\subset\subset\Omega$
there is a constant $C>0$ such that
\beqsn
\|P(D)^jf\|_{2,K}\leq C^{j+1}(j!)^m,
\eeqsn
where $m$ is the order of the operator and $\|\cdot\|_{2,K}$ is the $L^2$ norm
on $K$. This result was generalized for elliptic operators with variable
analytic
coefficients by Kotake and Narasimhan~\cite[Theorem 1]{KN}. Later this
result was
extended to the setting of Gevrey functions by Newberger and Zielezny~\cite{NZ}
and completely characterized by M\'etivier~\cite{metivier1978propiete}
(see also \cite{Z1}). Spaces of Gevrey type given by the iterates of a
differential operator are called {\em generalized Gevrey classes} and were used
by Langenbruch~\cite{langenbruch1979P,langenbruch1979Fortsetzung,langenbruch1985on,langenbruch1987bases} for different purposes. For more references about
generalized Gevrey classes and the microlocal version of the problem see
\cite{BJJ}.

More recently, Juan Huguet~\cite{juan-huguet2010iterates} extended the
results of Komatsu~\cite{K1}, Newberger and Zielezny~\cite{NZ} and
M\'etivier~\cite{metivier1978propiete} to the setting of non-quasianalytic
classes in the sense of Braun, Meise and Taylor~\cite{BMT}. In
\cite{juan-huguet2010iterates}, Juan Huguet introduced the generalized
spaces of ultradifferentiable functions $\E_{*}^P(\Omega)$ on an open subset
$\Omega$ of $\R^{n}$ for a fixed linear partial differential operator $P$
with constant coefficients, and proved that these spaces are complete if
and only if  $P$ is hypoelliptic. Moreover Juan Huguet showed that, in this
case, the spaces are nuclear. Later, the same author in \cite{J} established
a Paley-Wiener theorem for the classes $\E^{P}_{*}(\Omega)$, again under
the hypothesis of the hypoellipticity of $P$.

In order to remove the assumption on the hypoellipticity of the operator,
we considered in \cite{BJ} a different setting of ultradifferentiable functions,
following the ideas of \cite{BCM}.

The structure of the paper is as follows. We begin in Section 2 with some
notation and preliminaries. First, we introduce the classes of ultradifferentiable functions. Then we complete some partially known results on linear
partial differential operators with constant coefficients regarding
$\omega$-regularity that we will use in the last section. In Section 3 we
use Paley-Wiener theorems in \cite{J} to characterize the wave front set
$\WF_*^P(u)$ introduced in \cite{BJJ} (see Corollaries \ref{corWFB} and
\ref{corWFR}). The main tools to establish this characterization are
\cite[Proposition 17]{BJJ}, in which we proved that the product of a suitable
Gevrey function and a function in $\E^P_*(\Omega)$ is still in
$\E^P_*(\Omega)$ (observe that $\E^P_*(\Omega)$ is not an algebra
for pointwise multiplication in general), and the application of  pseudodifferential
operators defined by symbols supported in a given cone to the description
of the wave front set in Theorem~\ref{thm1} (see \cite[Proposition 3.4.4]{R}
for the corresponding result in Gevrey classes).  In the last section,
Section 4,
we give some applications and examples, in particular to operators with
variable
coefficients and constant strength. For this purpose we employ
some known results on
$\omega$-micro-regularity of operators with constant strength
(see \cite{FGJ1,FGJ}).

\section{Notation and preliminaries}

Let us recall from \cite{BMT} the definitions of weight functions $\omega$
and of the spaces of ultradifferentiable functions
of Beurling and Roumieu type:

\begin{Def}
\label{defweight}
A non-quasianalytic {\em weight function} is a continuous increasing function
$\omega:\ [0,+\infty[\to[0,+\infty[$ with the following properties:
\begin{itemize}
\item[$(\alpha)$]
$\exists\ L>0$ s.t. $\omega(2t)\leq L(\omega(t)+1)\quad\forall t\geq0$;
\item[$(\beta)$] $\int_1^{+\infty}\frac{\omega(t)}{t^2}dt<+\infty,$

\item[$(\gamma)$]
$\log(t)=o(\omega(t))$ as $t\to+\infty$;
\item[$(\delta)$]
$\varphi_\omega:\ t\mapsto\omega(e^t)$ is convex.
\end{itemize}

\end{Def}

Normally, we will denote $\varphi_{\omega}$ simply by $\varphi$.

For a weight function $\omega$ we define
$\overline{\omega}:\C^n\rightarrow[0,+\infty[$ by
$\overline{\omega}(z):=\omega(|z|)$ and again we denote this function by
$\omega$.

The {\em Young conjugate} $\varphi^*:\ [0,+\infty[\to
[0,+\infty[$ is defined by
$$
\varphi^*(s):=\sup_{t\geq0}\{st-\varphi(t)\}.
$$
There is no  loss of generality to assume that $\omega$ vanishes on $[0,1]$.
Then $\varphi^*$ has only non-negative values, it is convex, $\varphi^*(t)/t$
is increasing and tends to $\infty$ as $t\rightarrow \infty$, and
$\varphi^{**}=\varphi$.

\begin{Ex}{\rm The following functions are, after a
change
in some interval $[0,M]$,  examples of weight functions:\\
\noindent (i) $\omega (t)=t^d$ for $0<d<1.$ \\ \noindent (ii)
$\omega (t)=\left(\log(1+t)\right)^s$, $s>1.$ \\ \noindent (iii)
$\omega(t)=t(\log(e+t))^{-\beta}$, $\beta>1.$ \\ \noindent (iv)
$\omega(t)=\exp(\beta (\log(1+t))^{\alpha})$, $0<\alpha<1.$}
\end{Ex}
In what follows, $\Omega$ denotes an arbitrary subset of $\R^n$ and
$K\subset\subset\Omega$ means that $K$ is a compact subset in $\Omega$.

\begin{Def}\label{ultradifclases}
{\rm Let $\omega$  be a weight function.\\

(a) For a compact subset $K$ in $\R^n$ which coincides with the closure of
its interior and $\lambda>0$, we define the seminorm
$$
p_{K,\lambda}(f):=\displaystyle\sup_{x\in K}\sup_{\alpha\in
    \N_0^n}\left|f^{(\alpha)}(x)\right|
\exp\left(-\lambda\varphi^*\left(\frac{|\alpha|}{\lambda}\right)\right),
$$
where $\N_0:=\N\cup\{0\}$,
and set
$$\mathcal{E}_{\omega}^{\lambda}(K):= \{ f \in C^\infty(K):
p_{K ,\lambda}(f)<\infty\},$$
which is a Banach space endowed with the $p_{K,\lambda}(\cdot)$-topology.\\

(b) For an open subset $\Omega$ in $\R^n$, the class of
\emph{$\omega$-ultradifferentiable functions of Beurling type} is defined by
$$\mathcal{E}_{(\omega)}(\Omega):= \{ f \in C^\infty(\Omega):
p_{K ,\lambda}(f)<\infty , \, \mbox{for every } \, \,
K\subset \subset \Omega \, \, \mbox{and every}\,  \lambda
>0\}.$$
The topology of this space is
$$ \mathcal{E}_{(\omega)}(\Omega)=
\proj_{\stackrel{\longleftarrow}{K\subset\subset\Omega}}
\proj_{\stackrel{\longleftarrow}{\lambda>0}}
\mathcal{E}_{\omega}^{\lambda}(K),$$
and one can show that $\mathcal{E}_{(\omega)}(\Omega)$ is a Fr\'{e}chet space.\\

(c) For a compact subset $K$ in $\R^n$ which coincides with the closure of
its interior and $\lambda>0$, set
$$\mathcal{E}_{\{\omega\}}(K)=\{ f \in C^\infty(K):\mbox{ there exists }m\in\mathbb{N}\mbox{ such that }
p_{K ,\frac{1}{m}}(f)<\infty\}, $$
This space is the strong dual of a nuclear Fr\'{e}chet space (i.e.,
a (DFN)-space) if it is endowed with its natural inductive limit topology,
that is,
$$\mathcal{E}_{\{\omega\}}(K)=
\ind_{\stackrel{\longrightarrow}{m\in\mathbb{N}}}\mathcal
{E}_{\omega}^{\frac{1}{m}}(K).$$

(d) For an open subset $\Omega$ in $\R^n$, the class of
{\it $\omega$-ultradifferentiable functions of Roumieu type} is defined by:
$$\mathcal{E}_{\{\omega\}}(\Omega):= \{ f \in
C^\infty(\Omega):\,\forall  K\subset \subset \Omega\mbox{ } \exists \lambda
>0\mbox{ such that }  p_{K,\lambda}(f)<\infty\}.$$
Its topology is the following
$$
\mathcal{E}_{\{\omega\}}(\Omega)=
\proj_{\stackrel{\longleftarrow}{K\subset\subset\Omega}}\mathcal{E}_{\{\omega\}}(K),
$$
that is, it is endowed with the topology of the projective limit of the
spaces $\mathcal{E}_{\{\omega\}}(K)$ when $K$ runs the compact subsets of
$\Omega$. This is a complete PLS-space, that is, a complete space which is
a projective limit of LB-spaces (i.e., a countable inductive limit of
Banach spaces) with compact linking maps in the (LB)-steps. Moreover,
$\mathcal{E}_{\{\omega\}}(\Omega)$ is also a nuclear and reflexive locally
convex space. In particular, $\mathcal{E}_{\{\omega\}}(\Omega)$ is an
ultrabornological (hence barrelled and bornological) space. }
\end{Def}

The elements of
${\mathcal E}_{(\omega )}(\Omega)$ {\rm (}resp. ${\mathcal
E}_{\{\omega\}}(\Omega)${\rm )} are called ultradifferentiable functions
of Beurling type {\rm (}resp. Roumieu type{\rm )} in $\Omega.$

In the case that $\omega (t):=
t^d$ ($0<d<1$), the corresponding Roumieu class is the Gevrey class with
exponent $1/d.$ In the limit case $d=1$, not included in our setting, the
corresponding Roumieu class ${\mathcal
E}_{\{\omega\}}(\Omega)$ is the space of real analytic functions on $\Omega$
whereas the Beurling class ${\mathcal
E}_{(\omega)}({\mathbb R}^n)$ gives the entire functions.

 If a statement holds in the Beurling and the Roumieu case then we will use
the notation $\mathcal{E}_{*}(\Omega)$. It means that in all
cases * can be replaced either by $(\omega)$ or $\{\omega\}$.

 For a compact set $K$ in $\R^n$, define
$$
\mathcal{D}_{*}(K):=\{f\in\mathcal{E}_{*}(\R^n):\mbox{supp}f \subset K\},
$$
endowed with the induced topology.  For an open set $\Omega$ in $\R^n$, define
$$
\mathcal{D}_{*}(\Omega):=\ind_{\stackrel{\longrightarrow}{K\subset\subset\Omega}}
\mathcal{D}_{*}(K).
$$

Following \cite{juan-huguet2010iterates}, we consider smooth functions in an
open set $\Omega$ such that there exists $C>0$ verifying for each
$j \in {\N}_{0}:=\N\cup \{0\},$
$$
\|P^j(D)f\|_{2,K}\leq C\exp\left(\lambda\varphi^*(\frac{jm}{\lambda})\right),
$$
where $K$ is a compact subset in $\Omega$, $\|\cdot\|_{2,K}$ denotes the
${L}^2$-norm on $K$ and $P^j(D)$ is the $j$-th iterate of the partial
differential operator $P(D)$ of order $m$, i.e.,
 $$P^j(D)=P(D)\underbrace{\circ\cdot\cdot\cdot\circ}_{j}P(D).$$
 If $j=0$, then $P^0(D)f=f.$

Given a polynomial $P\in\mathbb{C}[z_1,\ldots,z_n]$ of degree $m$,
$P(z)=\sum\limits_{|\alpha|\leq m}a_{\alpha}z^{\alpha},$
the partial  differential operator $P(D)$ is defined as
$
P(D)=\sum_{|\alpha|\leq m}a_{\alpha}D^{\alpha}$, where $D=\frac{1}{i}\partial.$

The spaces of ultradifferentiable functions with respect to the successive
iterates of $P$ are defined as follows.

 Let $\omega$ be a weight function. Given a polynomial $P$, an open set
$\Omega$ of
$\R^n$, a compact subset $K\subset\subset\Omega$ and $\lambda>0$,
we define the seminorm
\begin{equation}
\|f\|_{K,\lambda}:=\sup_{j\in\mathbb{N}_{0}}\|P^j(D)f\|_{2,K}\exp\left(-\lambda
\varphi^*(\frac{jm}{\lambda})\right)\label{generalized-seminorm}
\end{equation}
and set
$$
\mathcal{E}_{P,\omega}^{\lambda}(K)=\{ f\in\mathcal{C}^{\infty}(K):\mbox{ }
\|f\|_{K,\lambda}<+\infty\}.
$$

It is a normed space endowed with the $\|\cdot\|_{K,\lambda}$-norm.

The space of {\it ultradifferentiable functions of Beurling type with
respect to the iterates of $P$} is:
$$
\mathcal{E}^P_{(\omega)}(\Omega)=\{ f\in\mathcal{C}^{\infty}(\Omega):
\mbox{ }\|f\|_{K,\lambda}<+\infty\mbox{ for each }K\subset\subset\Omega
\mbox{ and }\lambda>0\},
$$
{\rm endowed with the topology given by}
 $$\mathcal{E}^P_{(\omega)}(\Omega):=
\proj_{\stackrel{\longleftarrow}{K\subset\subset\Omega}}
\proj_{\stackrel{\longleftarrow}{\lambda>0}}\mathcal
{E}_{P,\omega}^{\lambda}(K). $$

{\rm If $\{K_n\}_{n\in\mathbb{N}}$ is a compact
exhaustion of $\Omega$ we have}
 $$ \mathcal{E}^P_{(\omega)}(\Omega)=\proj_{\stackrel{\longleftarrow
}{n\in\mathbb{N}} }\proj_{ \stackrel{\longleftarrow}
 {k\in\mathbb{N}}}\mathcal{E}_{P,\omega}^{k}(K_n)=
\proj_{\stackrel{\longleftarrow}{n\in\mathbb{N}}}\mathcal
{E}_{P,\omega}^{n}(K_n).$$

This is a metrizable locally convex topology defined by the
fundamental system of seminorms
$\left\{\|\cdot\|_{K_{n},n}\right\}_{n\in\mathbb{N}}$. 

The space of {\it ultradifferentiable functions of Roumieu type
with respect to the iterates of $P$} is defined by:
$$
\mathcal{E}^P_{\{\omega\}}(\Omega)=\{ f\in\mathcal{C}^{\infty}(\Omega):
\mbox{ }\forall K\subset\subset\Omega\mbox{ }\exists\lambda>0
\mbox{ such that }
\|f\|_{K,\lambda}<+\infty\}.
$$
{\rm Its topology is defined by}
 $$\mathcal{E}^P_{\{\omega\}}(\Omega):=
\proj_{\stackrel{\longleftarrow}{K\subset\subset\Omega}}
\ind_{\stackrel{\longrightarrow}{\lambda>0}}\mathcal
{E}_{P,\omega}^{\lambda}(K). $$

As in the Gevrey case, we call these classes
{\it generalized non-quasianalytic classes.} We observe that
in comparison with the spaces of generalized non-quasianalytic
classes as defined in \cite{juan-huguet2010iterates} we add here $m$
as a factor inside $\varphi^*$ in \eqref{generalized-seminorm}, where
$m$ is the order of the operator $P$, which does not change the properties
of the classes and will simplify the notation in the following.

The inclusion map $\E_\ast(\Omega)\hookrightarrow \E^P_\ast(\Omega)$ is
continuous (see \cite[Theorem 4.1]{juan-huguet2010iterates}). The space
$\mathcal{E}^P_{*}(\Omega)$ is complete if and only if $P$ is hypoelliptic
(see \cite[Theorem 3.3]{juan-huguet2010iterates}). Moreover, under a mild
condition on $\omega$ introduced by Bonet, Meise and Melikhov
\cite{bonet_meise_melikhov2007a}, $\mathcal{E}^P_{*}(\Omega)$  coincides
with the class of ultradifferentiable functions $\mathcal{E}_{*}(\Omega)$ if
and only if $P$ is elliptic (see \cite[Theorem 4.12]{juan-huguet2010iterates}).

We denote by
\beqsn
\widehat{f}(\xi):=\int e^{-i\langle x,\xi\rangle}f(x)dx
\eeqsn
the classical Fourier transform $\F(f)$.

Now, let $P(D)=\sum_{|\alpha|\leq m}c_\alpha D^\alpha$
 be a linear partial differential operator with constant
coefficients, where $D=-i\partial$. We recall the notion of
hypoellipticity in the $C^\infty$ class:
$P(D)$ is hypoelliptic in $\Omega\subseteq\R^n$ if $P(D)u\in C^\infty(\Omega)$
implies $u\in C^\infty(\Omega)$. In this case we also say that the polynomial
$P(\xi)=\sum_{|\alpha|\leq m}c_\alpha \xi^\alpha$ is hypoelliptic.

We set
\beqsn
V=V(P):=\{\zeta\in\C^n:\ P(\zeta)=0\}
\eeqsn
and consider the distance from $\xi\in\R^n$ to $V$:
\beqsn
d(\xi):=\inf_{\zeta\in V}|\xi-\zeta|,\qquad\xi\in\R^n.
\eeqsn
From \cite[Thms 11.1.1 and 11.1.3]{H} (see also
\cite[Prop. 2.2.1]{R}) we recall the following characterization of
hypoellipticity, that will be useful in the following:
\begin{Th}
\label{thH}
Let $P(D)$ be a linear partial differential operator with constant
coefficients.
The following properties are equivalent for $P$ to be hypoelliptic:
\begin{itemize}
\item[(1)]
For every open set $\Omega\subseteq\R^n$ and $u\in\D'(\Omega)$
\beqsn
\WF(u)=\WF(P(D)u).
\eeqsn
\item[(2)]
For every open set $\Omega\subseteq\R^n$ and $u\in\D'(\Omega)$
\beqsn
\singsupp u=\singsupp P(D)u.
\eeqsn
\item[(3)]
$P$ is homogeneous hypoelliptic, i.e. if $\Omega$ is open in $\R^n$
and $u\in\D'(\Omega)$ then $P(D)u=0$ implies $u\in C^\infty(\Omega)$.
\item[(4)]
$\ds \lim_{\afrac{\xi\in\R^n}{|\xi|\to+\infty}}\frac{D^\alpha P(\xi)}{P(\xi)}=0
\qquad\forall\alpha\neq0.$
\item[(5)]
$P(D)$ has a fundamental solution $E$ with $\singsupp E=\{0\}$.
\item[(6)]
$\ds\lim_{\afrac{\xi\in\R^n}{|\xi|\to+\infty}}d(\xi)=+\infty$.
\item[(7)]
$\ds\lim_{\afrac{\zeta\in V}{|\zeta|\to+\infty}}|\Im\zeta|=+\infty.$
\item[(8)]
There exist $C>0$ and a largest $0<c\leq 1$, with $c\in\Q$, such that
for all $\alpha\neq0$:
\beqsn
\left|\frac{D^\alpha P(\xi)}{P(\xi)}\right|\leq C|\xi|^{-|\alpha|c}
\qquad\xi\in\R^n,\ |\xi|\gg1.
\eeqsn
\item[(9)]
There exist $C>0$ and a largest $0<c\leq 1$, with $c\in\Q$, such that
\beqs
\label{defc}
d(\xi)\geq C|\xi|^c,\qquad\xi\in\R^n,\ |\xi|\gg1.
\eeqs
\end{itemize}

Here $\WF(u)$ and $\singsupp u$ denote the classical wave front set
and singular support of $u\in\D'(\Omega)$, as defined in \cite{H1}.
\end{Th}

\begin{Rem}
\label{rem22}
\begin{em}
By \cite[Lemma 11.1.4]{H} there exists a constant $C>0$ such that for
all polynomials $P$ of degree $\leq m$
\beqs
\label{2210R}
C^{-1}\leq d(\xi)\sum_{\alpha\neq0}
\left|\frac{D^\alpha P(\xi)}{P(\xi)}\right|^{\frac{1}{|\alpha|}}
\leq C\qquad\mbox{for}\ \xi\in\R^n\ \mbox{with}\ P(\xi)\neq0.
\eeqs
Therefore the constant $c$ at the exponent in $(8)$ and $(9)$ of
Theorem \ref{thH} coincide.

In particular, from $(8)$ with $|\alpha|=m$ we have that if $P$ is
hypoelliptic and of order $m$ then there exist $\delta,d>0$ such that
\beqs
\label{H1}
|P(\xi)|\geq\delta|\xi|^d\qquad\xi\in\R^n,\ |\xi|\gg1,
\eeqs
with $mc=d\leq m$.

Moreover, by \cite[Thm 3.1]{Hint}, there exists a smallest constant $\gamma:=\gamma(P)>0$, which depends on $P$ and will be relevant later
such that
\beqs
\label{gamma}
|D^\alpha P(\xi)|^2\leq C(1+|P(\xi)|^2)^{1-\frac{|\alpha|}{\gamma}}\qquad\forall
\xi\in\R^n\qquad\forall\alpha\in\N_0^n,
\eeqs
for some $C>0$. Note that $m\leq\gamma\leq m/c$, since $b:=1/\gamma$ has been
determined in Theorem 3.1 of \cite{Hint} applying the Tarski-Seidenberg
theorem to
\beqs
\label{b1}
M(\lambda):=\sup_{|P(\xi)|=\lambda}|\grad P(\xi)|=A\lambda^{1-b}(1+o(1)),
\eeqs
and $(8)$ of Theorem \ref{thH} implies
\beqs
\label{b2}
M(\lambda)\leq A'\lambda^{1-\frac cm},\qquad\lambda\gg1
\eeqs
for some $A,A'>0$, if $P$ has order $m$;
\eqref{b1} and \eqref{b2} imply $b\geq c/m$ and hence $m\leq\gamma\leq m/c$
($\gamma\geq m$ by \eqref{gamma} with $|\alpha|=m$).
\end{em}
\end{Rem}

We want to generalize Theorem \ref{thH} to $\omega$-hypoellipticity,
for a weight function $\omega$.
\begin{Def}
A linear partial differential operator $P(D)$ with constant
coefficients is said to be $(\omega)$-hypoelliptic
($\{\omega\}$-hypoellpitic) if every solution $u\in\D'(\R^n)$
of $P(D)u=f$ is in $\E_{(\omega)}(\R^n)$ whenever $f\in\E_{(\omega)}(\R^n)$
($u\in\E_{\{\omega\}}(\R^n)$ whenever $f\in\E_{\{\omega\}}(\R^n)$).
\end{Def}

We have the following characterization of $(\omega)$-hypoellipticity
(Beurling case):
\begin{Th}
\label{thhypoB}
Let $P(D)$ be a linear partial differential operator with
constant coefficients and $\omega$ a non-quasianalytic weight function.
The following conditions are equivalent:
\begin{itemize}
\item[(1)]
$P(D)$ is $(\omega)$-hypoelliptic.
\item[(2)]
$\ds\lim_{\afrac{\zeta\in V}{|\zeta|\to+\infty}}\frac{|\Im\zeta|}{\omega(\zeta)}
=+\infty.$
\item[(3)]
$\ds\lim_{\afrac{\xi\in\R^n}{|\xi|\to+\infty}}\frac{\omega(\xi)}{d(\xi)}=0.$
\item[(4)]
$\ds\lim_{\afrac{\xi\in\R^n}{|\xi|\to+\infty}}
\frac{\omega(\xi)^{|\alpha|}|D^\alpha P(\xi)|}{|P(\xi)|}=0
\qquad\forall\alpha\neq0$.
\end{itemize}
\end{Th}

\begin{proof}
$(1)\Leftrightarrow(2)$ is proved in \cite{BFM}.

$(3)\Rightarrow(2)$: If $\zeta\in V$ and $|\Im\zeta|\leq|\Re\zeta|$, then from
property $(\alpha)$ of $\omega$ we have that
\beqsn
\frac{|\Im\zeta|}{\omega(\zeta)}\geq\frac{d(\Re\zeta)}{L(\omega(\Re\zeta)+1)}
\longrightarrow+\infty
\eeqsn
by $(3)$.

If $\zeta\in V$ and $|\Re\zeta|\leq|\Im\zeta|$ then
\beqsn
\frac{|\Im\zeta|}{\omega(\zeta)}\geq
\frac12\frac{|\Re\zeta|+|\Im\zeta|}{\omega(\zeta)}\longrightarrow+\infty
\eeqsn
since $\omega(t)=o(t)$.

$(2)\Rightarrow(3)$: For every fixed $\xi\in\R^n$ take $\zeta\in V$ with
$|\zeta-\xi|\leq 2d(\xi)$. Take $|\xi|$ large enough.

If $|\zeta-\xi|\leq|\zeta|$, then $|\xi|\leq2|\zeta|$ and
\beqsn
\frac{d(\xi)}{\omega(\xi)}\geq\frac12\frac{|\zeta-\xi|}{\omega(\xi)}
\geq\frac12\frac{|\Im\zeta|}{L(\omega(\zeta)+1)}\longrightarrow+\infty
\eeqsn
because of $(2)$.

In $|\zeta-\xi|\geq|\zeta|$, then $(3)$ follows from the inequality
\beqsn
\frac{d(\xi)}{\omega(\xi)}\geq\frac12
\max\left\{\frac{|\xi|}{\omega(\xi)}-\frac{|\zeta|}{\omega(\xi)},
\frac{|\zeta|}{\omega(\xi)}\right\},
\eeqsn
since $\omega(t)=o(t)$.

$(3)\Leftrightarrow(4)$: follows from \eqref{2210R}.
\end{proof}

\begin{Ex}
\begin{em}
Theorem \ref{thhypoB} shows, for example, that the heat operator
$P=\partial_t-\Delta_x$ is not $(t^{1/2})$-hypoelliptic, since
\beqsn
\frac{|(\tau,\xi)|^{2a}|D_{x_j}^2P(\tau,\xi)|}{|P(\tau,\xi)|}
=\frac{(\tau^2+|\xi|^2)^a\cdot2}{|i\tau+|\xi|^2|}
=\frac{2(\tau^2+|\xi|^2)^a}{\sqrt{\tau^2+|\xi|^4}}\not\!\!\longrightarrow0
\eeqsn
for $\tau=|\xi|\to+\infty$ if $2a\geq1$.
On the other hand, it is well known that the heat operator is
$\{t^{1/2}\}$-hypoelliptic (as can be seen also by Theorem \ref{thhypoR} below).
\end{em}
\end{Ex}

In the Roumieu case we have the following theorem
of characterization of $\{\omega\}$-hypoellipticity, which generalizes
Proposition 2.2.1 of \cite{R} (where a similar result is given for the
Gevrey classes):
\begin{Th}
\label{thhypoR}
Let $P(D)$ be a linear partial differential operator with
constant coefficients and $\omega$ a non-quasianalytic weight function.
The following conditions are equivalent:
\begin{itemize}
\item[(1)]
$P(D)$ is $\{\omega\}$-hypoelliptic.
\item[(2)]
$\ds\liminf_{\zeta\in V,|\zeta|\to+\infty}\frac{|\Im\zeta|}{\omega(\zeta)}>0.$
\item[(3)]
There exists $c>0$ such that
\beqsn
\omega(\zeta)\leq c(1+|\Im\zeta|)\qquad\forall\zeta\in V.
\eeqsn
\item[(4)]
There exist $c,C>0$ such that
\beqsn
\omega(\xi)\leq Cd(\xi)\qquad\mbox{for}\ \xi\in\R^n,\ |\xi|>c.
\eeqsn
\item[(5)]
There exist $c,C>0$ such that
\beqsn
|D^\alpha P(\xi)|\leq C|P(\xi)|\omega(\xi)^{-|\alpha|}
\qquad\mbox{for}\ \xi\in\R^n,\ |\xi|>c,\ \alpha\in\N_0^n.
\eeqsn
\end{itemize}
\end{Th}

\begin{proof}
$(1)\Leftrightarrow(2)$ is proved in \cite{BFM}.

$(2)\Leftrightarrow(3)$: it is easy to check.

$(3)\Rightarrow(4)$:
For every fixed $\xi\in\R^n$ take $\zeta\in V$ with $|\zeta-\xi|\leq 2d(\xi)$.

If $|\zeta|\leq|\xi-\zeta|$ then $|\xi|\leq 2|\xi-\zeta|$ and for $|\xi|$
large enough
\beqsn
\omega(\xi)\leq\omega(2|\xi-\zeta|)\leq c|\xi-\zeta|\leq 2cd(\xi)
\eeqsn
for some $c>0$ since $\omega(t)=o(t)$.

If $|\xi-\zeta|\leq|\zeta|$ then $|\xi|\leq2|\zeta|$ and, by
property $(\alpha)$ of $\omega$ and $(3)$,
\beqsn
\omega(\xi)\leq L(\omega(\zeta)+1)\leq L'(|\Im\zeta|+1)
\leq L'(|\zeta-\xi|+1)
\leq L''d(\xi),
\eeqsn
for some $L',L''>0$, since $(3)$ implies that $d(\xi)\to+\infty$
(see \cite[Prop. 2.2.1]{R}).

$(4)\Rightarrow(3)$:
If $|\Im\zeta|\leq|\Re\zeta|$, then property $(\alpha)$ of $\omega$ and $(4)$
imply that
\beqsn
\omega(\zeta)\leq L(\omega(\Re\zeta)+1)\leq L'(d(\Re\zeta)+1)\leq
L'(|\Im\zeta|+1)
\eeqsn
for some $L'>0$.

If $|\Re\zeta|\leq|\Im\zeta|$ then, by property $(\alpha)$ of $\omega$,
\beqsn
\omega(\zeta)\leq L(\omega(|\Im\zeta|)+1)\leq L'(|\Im\zeta|+1)
\eeqsn
for some $L'>0$ since $\omega(t)=o(t)$.

$(4)\Leftrightarrow(5)$: It is straightforward because of \eqref{2210R}.
\end{proof}

\begin{Rem}
\label{rem27}
\begin{em}
From Theorems \ref{thhypoB} and \ref{thhypoR} we immediately get the
well-known result that $(\omega)$-hypoellipticity implies
$\{\omega\}$-hypoellipticity and they both imply hypoellipticity by
Theorem~\ref{thH}.
\end{em}
\end{Rem}

\section{Characterization of $\omega$-micro-hypoellipticity with
respect to the iterates of an operator}

Let $\omega$ be a non-quasianalytic weight function and $P(D)$ a linear partial
differential operator with constant coefficients. We want to characterize
the functions in the class $\E_*^P$, for
$*=(\omega)$ or $\{\omega\}$, by means of the Fourier transform.
To this aim we use a Paley-Wiener theorem for these spaces
of functions, that we borrow from \cite{J}.
Since our spaces $\E^P_*(\Omega)$ are slightly different from the
analogous ones defined in \cite{J}, our Paley-Wiener theorem
is only sligthly different from the one of \cite{J}, and therefore
we shall present the suitable statement here, omiting the proof  (see \cite[Lemma 3.1]{J}).

\begin{Lemma}
\label{lemmaPW}
Let $\omega$ be a non-quasianalytic
weight function, $P$ a polynomial of degree $m$, $K$ a
compact convex subset of $\R^n$ and $f\in\D(\R^n)$ with $\supp f\subset K$.
Then the following statements are equivalent:
\begin{itemize}
\item[(1)]
there exists $\lambda>0$ such that
\beqsn
\int_{\R^n}|\widehat{f}(\xi)|^2e^{\lambda\omega(|P(\xi)|^{1/m})}d\xi<+\infty;
\eeqsn
\item[(2)]
there exists $\lambda,C>0$ such that
\beqsn
\|P^j(D)f\|_{2,\R^n}\leq Ce^{\lambda\varphi^*\left(\frac{jm}{\lambda}\right)}
\qquad\forall j\in\N_0.
\eeqsn
\end{itemize}

To be more precise, if $(1)$ holds for some $\lambda_0>0$, then
$(2)$ holds for $\lambda=\lambda_0/2$ and for
$C:=(2\pi)^{-n/2}\left(\int_{\R^n}|\widehat{f}(\xi)|^2
e^{\lambda_0\omega(|P(\xi)|^{1/m})}d\xi\right)^{1/2}$.

Vice versa, if $(2)$ holds for some $\lambda>0$, then
\beqs
\label{PW3}
|\widehat{f}(\zeta)|\leq m(K)^{1/2}CD_{\lambda,\omega}
e^{H_K(\Im\zeta)-\frac\lambda2\omega(|P(\zeta)|^{1/m})}
\qquad\forall\zeta\in\C^n,
\eeqs
for some $D_{\lambda,\omega}>0$ depending on $\lambda$ and $\omega$,
where $m(K)$ is the Lebesgue measure of $K$ and $H_K(\cdot)$ the
supporting function of $K$; therefore $(1)$ holds for any
$\lambda'<\lambda$.
\end{Lemma}

We present also, similarly as in  \cite[Thms 3.3 and 3.4]{J}, the following
Paley-Wiener type theorem for functions in
\beqsn
\D^P_*(\R^n)=\{f\in\E^P_*(\R^n): f\in\D(\R^n)\},
\eeqsn
for $*=(\omega)$ or $\{\omega\}$:
\begin{Th}
\label{thPW}
Let $P(\xi)$ be a hypoelliptic polynomial of degree $m$ and
$\omega$ a non-qua\-si\-ana\-ly\-tic weight function.
If $f\in\D^P_{(\omega)}(\R^n)$ (resp. $f\in\D^P_{\{\omega\}}(\R^n)$)
then its Fourier-Laplace trans\-form $F(\zeta)=\widehat{f}(\zeta)$
is an entire function satisfying the following two conditions:
\begin{itemize}
\item[(i)]
there exist $C,A>0$ such that
\beqsn
|F(\zeta)|\leq Ce^{A|\zeta|}\qquad\forall\zeta\in\C^n;
\eeqsn
\item[(ii)]
for every $\lambda>0$ (resp. there exists $\lambda>0$):
\beqsn
\int_{\R^n}|F(\xi)|^2e^{\lambda\omega(|P(\xi)|^{\frac1m})}d\xi<+\infty.
\eeqsn
\end{itemize}
Vice versa, if $F$ is an entire function satisfying $(i)$ and $(ii)$, then
$F(\zeta)=\widehat{f}(\zeta)$ for some
$f\in\D^P_{(\omega)}(\R^n)$ (resp. $f\in\D^P_{\{\omega\}}(\R^n)$).
\end{Th}

Now, we have all the tools to prove the following theorems of
characterization of $\E_{(\omega)}^P$ and $\E_{\{\omega\}}^P$:
\begin{Th}
\label{thB}
Let $P(D)$ be a hypoelliptic linear partial differential
operator with constant coefficients, $\Omega$ an open subset of $\R^n$, $u\in\D'(\Omega)$
and $x_0\in\Omega$.
Let $\omega$ be a non-quasianalytic weight function such that
$\omega(t^\gamma)=o(\sigma(t))$, as $t\to+\infty$, where $\gamma$ is the
constant defined in \eqref{gamma} and $\sigma(t)=t^{1/s}$ for some
$s>1$.

The following conditions are equivalent:
\begin{itemize}
\item[(1)]
$u\in\E_{(\omega)}^P(U)$ for some neighborhood $U$ of $x_0$.
\item[(2)]
There exists $\{f_N\}_{N\in\N}\subset\E'(\Omega)$ such that
$f_N=P(D)^Nu$ in a neighborhood of $x_0$ and:
\beqs
\nonumber
&&\forall k\in\N, M\in\R\ \exists C_{k,M}>0 :\\
\label{B0}
&&|\widehat{f}_N(\xi)|\leq C_{k,M}e^{k\varphi^*(Nm/k)}(1+|\xi|)^M
\qquad\forall N\in\N,\,\xi\in\R^n.
\eeqs
\item[(3)]
There exists $\psi\in\D_{\{\sigma\}}(\Omega)$, with $\psi\equiv1$ in a
neighborhood of $x_0$, such that:
\beqs
\nonumber
&&\forall k\in\N\ \exists C_k>0 :\\
\label{B1}
&&|\widehat{\psi u}(\xi)|\leq C_ke^{-k\omega(|P(\xi)|^{1/m})}
\qquad\forall\xi\in\R^n.
\eeqs
\item[(4)]
There exists $\psi\in\D_{\{\sigma\}}(\Omega)$, with $\psi\equiv1$ in a
neighborhood of $x_0$, such that:
\beqs
\nonumber
&&\forall k\in\N,\ell>0\ \exists C_{k,\ell}>0 :\\
\label{B2}
&&|P(\xi)|^N|\widehat{\psi u}(\xi)|\leq C_{k,\ell}e^{k\varphi^*(Nm/k)}
(1+|\xi|)^{-\ell}
\qquad\forall N\in\N_0,\,\xi\in\R^n.
\eeqs
\end{itemize}
\end{Th}

\begin{Th}
\label{thR}
Let $P(D)$ be a hypoelliptic linear partial differential
operator with constant coefficients, $\Omega$ an open subset of $\R^n$, $u\in\D'(\Omega)$
and $x_0\in\Omega$.
Let $\omega$ be a non-quasianalytic weight function such that
$\omega(t^\gamma)=o(\sigma(t))$, as $t\to+\infty$, where $\gamma$ is the
constant defined in \eqref{gamma} and $\sigma(t)=t^{1/s}$ for some
$s>1$.

Then the following conditions are equivalent:
\begin{itemize}
\item[\{1\}]
$u\in\E_{\{\omega\}}^P(U)$ for some neighborhood $U$ of $x_0$.
\item[\{2\}]
There exists $\{f_N\}_{N\in\N}\subset\E'(\Omega)$ such that
$f_N=P(D)^Nu$ in a neighborhood of $x_0$ and:
\beqs
\nonumber
&&\exists k\in\N, \forall M\in\R\ \exists C_M>0:\\
\label{R0}
&&|\widehat{f}_N(\xi)|\leq C_Me^{\frac1k\varphi^*(Nmk)}(1+|\xi|)^M
\qquad\forall N\in\N,\,\xi\in\R^n.
\eeqs
\item[\{3\}]
There exists $\psi\in\D_{\{\sigma\}}(\Omega)$, with $\psi\equiv1$ in a
neighborhood of $x_0$, such that:
\beqs
\nonumber
&&\exists k\in\N, C>0:\\
\label{R1}
&&|\widehat{\psi u}(\xi)|\leq Ce^{-\frac1k\omega(|P(\xi)|^{1/m})}
\qquad\forall\xi\in\R^n.
\eeqs
\item[\{4\}]
There exists $\psi\in\D_{\{\sigma\}}(\Omega)$, with $\psi\equiv1$ in a
neighborhood of $x_0$, such that:
\beqs
\nonumber
&&\exists k\in\N,\forall\ell>0\ \exists C_\ell>0:\\
\label{R2}
&&|P(\xi)|^N|\widehat{\psi u}(\xi)|\leq C_\ell e^{\frac1k\varphi^*(Nmk)}
(1+|\xi|)^{-\ell}
\qquad\forall N\in\N_0,\,\xi\in\R^n.
\eeqs
\end{itemize}
\end{Th}

Before proving Theorems \ref{thB} and \ref{thR} we need the following
lemma:
\begin{Lemma}
\label{lemmaBR}
Let $\Gamma\subseteq \R^n$ be a cone, $P(D)$ a linear partial
differential operator of order $m$ with constant coefficients,
$\Omega$ an open subset of $\R^n$, $u\in\D'(\Omega)$,
$\psi\in \D(\Omega)$. Then:
\begin{itemize}
\item[(1)]\underline{Beurling case}. The following two conditions
are equivalent:
\begin{itemize}
\item[(1.a)]
for every $k\in\N$ there exists $C_k>0$ such that
\beqsn
|\widehat{\psi u}(\xi)|\leq C_ke^{-k\omega(|P(\xi)|^{1/m})}
\qquad\forall\xi\in \Gamma;
\eeqsn
\item[(1.b)]
for every $k\in\N$, $\ell>0$ there exists $C_{k,\ell}>0$ such that
\beqsn
|P(\xi)|^N|\widehat{\psi u}(\xi)|\leq C_{k,\ell}e^{k\varphi^*(Nm/k)}
(1+|\xi|)^{-\ell}
\qquad\forall N\in\N_0,\,\xi\in \Gamma.
\eeqsn
\end{itemize}
\item[(2)]\underline{Roumieu case}. The following two conditions
are equivalent:
\begin{itemize}
\item[(2.a)]
there exist $k\in\N$ and $C>0$ such that
\beqsn
|\widehat{\psi u}(\xi)|\leq Ce^{-\frac1k\omega(|P(\xi)|^{1/m})}
\qquad\forall\xi\in \Gamma;
\eeqsn
\item[(2.b)]
there exists $k\in\N$ such that for every $\ell>0$ there is
$C_\ell>0$ s.t.
\beqsn
|P(\xi)|^N|\widehat{\psi u}(\xi)|\leq C_\ell e^{\frac1k\varphi^*(Nmk)}
(1+|\xi|)^{-\ell}
\qquad\forall N\in\N_0,\,\xi\in \Gamma.
\eeqsn
\end{itemize}
\end{itemize}
\end{Lemma}

\begin{proof}[Proof of Lemma \ref{lemmaBR}]
\underline{Beurling case}. $(1.a)\Rightarrow(1.b)$:
Since $P$ is hypoelliptic, by \eqref{H1}, $(1.a)$,
\cite[Lemma 16(i)]{BJJ} and the convexity of
$\varphi^*$, we have, for all $\ell>0$ and $N\in\N_0$,
\beqsn
|\xi|^\ell|P(\xi)|^N|\widehat{\psi u}(\xi)|\leq&&
\delta^{-\ell/d}|P(\xi)|^{\frac{\ell}{d}+N}|\widehat{\psi u}(\xi)|\\
\leq&&\delta^{-\ell/d}(|P(\xi)|^{\frac1m})^{\frac{\ell m}{d}+mN}C_{2k}
e^{-2k\omega(|P(\xi)|^{\frac1m})}\\
\leq&&C_{k,\ell}e^{k\varphi^*\left(\frac{Nm}{k}\right)}.
\eeqsn
$(1.b)\Rightarrow(1.a)$:
Assume first that $|P(\xi)|>1$. By $(1.b)$ with $\ell=m$ we have that
for every $k$ there exists a constant $C_{k,m}=C_k>0$ such that
\beqsn
|\xi|^m|P(\xi)|^Ne^{-k\varphi^*\left(\frac{Nm}{k}\right)}
|\widehat{\psi u}(\xi)|\leq C_k\qquad
\forall N\in\N,\,\xi\in\Gamma.
\eeqsn
Since $|P(\xi)|\leq c|\xi|^m$ for some $c>0$, we thus have that
\beqs
\label{L1}
\sup_{N\in\N_0}\left\{(|P(\xi)|^{1/m})^{(N+1)m}
e^{-k\varphi^*\left(\frac{Nm}{k}\right)}\right\}
|\widehat{\psi u}(\xi)|\leq C'_k
\eeqs
for some $C'_k>0$. But for all $s>0$ there exists $N\in\N_0$ such that
$Nm\leq s<(N+1)m,$ so that
\beqsn
\sup_{N\in\N_0}\left\{(|P(\xi)|^{1/m})^{(N+1)m}
e^{-k\varphi^*\left(\frac{Nm}{k}\right)}\right\}
\geq&&\sup_{s>0}\left\{(|P(\xi)|^{1/m})^s
e^{-k\varphi^*\left(\frac sk\right)}\right\}\\
=&&\exp\{k\varphi(\log|P(\xi)|^{1/m})\}
=e^{k\omega(|P(\xi)|^{1/m})}.
\eeqsn
Substituting in \eqref{L1} we have that
\beqsn
e^{k\omega(|P(\xi)|^{1/m})}|\widehat{\psi u}(\xi)|\leq C'_k
\qquad\forall\xi\in\Gamma \ \mbox{with}\ |P(\xi)|>1.
\eeqsn

If $|P(\xi)|\leq1$ then $\omega(|P(\xi)|^{1/m})\equiv0$ and the thesis
is trivial.

\underline{Roumieu case}. It is similar to the Beurling case.
\end{proof}

\begin{proof}[Proof of Theorem \ref{thB}]
$(1)\Leftrightarrow(2)$ was proved in
\cite[Prop. 6]{BJJ}.

$(1)\Rightarrow(3)$: Let $u\in\E_{(\omega)}^P(U)$ and take $\psi\in
\D_{\{\sigma\}}(U)$ with $\psi\equiv1$ in a neighborhood $V\subset U$ of
$x_0$.
Since $\omega(t^\gamma)=o(\sigma(t))$, by Proposition 17 of \cite{BJJ}
we have that $\psi u\in\D_{(\omega)}^P(U)$.

By the Paley-Wiener Theorem \ref{thPW}, for all $\lambda>0$
there exists $C_\lambda>0$ such that
\beqsn
\int_{\R^n}|\widehat{\psi u}(\xi)|^2e^{4\lambda\omega(|P(\xi)|^{1/m})}d\xi
\leq C_\lambda
\eeqsn
and hence, by Lemma \ref{lemmaPW} for $K=\supp\psi$,
\beqsn
|\widehat{\psi u}(\zeta)|\leq m(K)^{1/2}C_\lambda D_{\lambda,\omega}
e^{H_K(\Im\zeta)-\lambda\omega(|P(\zeta)|^{1/m})}
\qquad\forall\zeta\in\C^n.
\eeqsn
For $\lambda=k\in\N$ and $\zeta=\xi\in\R^n$ we thus obtain \eqref{B1}.

$(3)\Rightarrow(1)$:
From \eqref{B1} it follows that for every $\lambda>0$
\beqsn
C_\lambda:=\left(\int_{\R^n}|\widehat{\psi u}(\xi)|^2
e^{2\lambda\omega(|P(\xi)|^{1/m})}d\xi\right)^{1/2}<+\infty.
\eeqsn
Therefore, by Lemma \ref{lemmaPW}, we obtain
\beqsn
\|P^j(D)(\psi u)\|_{2,\R^n}
\leq(2\pi)^{-n/2}C_\lambda e^{\lambda\varphi^*\left(\frac{jm}{\lambda}\right)}
\qquad\forall j\in\N_0.
\eeqsn
Since $\psi\equiv1$ in a neighborhood $U$ of $x_0$, then for every
compact $K\subset\subset U$
\beqsn
\|P^j(D)u\|_{2,K}=\|P^j(D)(\psi u)\|_{2,K}
\leq\|P^j(D)(\psi u)\|_{2,\R^n}
\leq C'_\lambda e^{\lambda\varphi^*\left(\frac{jm}{\lambda}\right)}
\qquad\forall j\in\N_0,
\eeqsn
for some $C'_\lambda>0$, i.e. $u\in\E_{(\omega)}^P(U)$.

$(3)\Leftrightarrow(4)$ follows from Lemma \ref{lemmaBR} for
$\Gamma=\R^n$.
\end{proof}

\begin{proof}[Proof of Theorem \ref{thR}]
It is analogous to that of Theorem \ref{thB} for $\lambda=1/k$.
\end{proof}

Let us now consider the wave front set with respect to the iterates of an
operator. We recall the following definition from \cite{BJJ}:
\begin{Def}
\label{defWF1}
Let $P(D)$ be a hypoelliptic linear partial differential operator with
constant coefficients, $\omega$ a non-quasinanlytic weight function,
$\Omega$ an open subset of $\R^n$, $u\in\D'(\Omega)$.
We say that a point $(x_0,\xi_0)\in\Omega\times(\R^n\setminus\{0\})$ is not
in the $(\omega)$-wave front set $\WF_{(\omega)}^P(u)$ (resp.
$\{\omega\}$-wave front set $\WF_{\{\omega\}}^P(u)$) with respect to
the iterates of $P$, if there are a neighborhood $U$ of $x_0$, an open
conic neighborhood $\Gamma$ of $\xi_0$ and a sequence $\{f_N\}_{N\in\N}
\subset\E'(\Omega)$ such that $(i)$ and $(ii)$ (resp. $(i)$ and $(iii)$)
of the following conditions hold:
\begin{itemize}
\item[(i)] $f_N=P(D)^Nu$ in $U$.
\item[(ii)]
\underline{Beurling case}:
\begin{itemize}
\item[(a)]
there exist $M,C>0$ such that for all $k\in\N$ there is $C_k>0$:
\beqs
\label{B4}
|\widehat{f}_N(\xi)|\leq C_k C^N(e^{\frac{k}{Nm}\varphi^*\left(\frac{Nm}{k}\right)}
+|\xi|)^{Nm}(1+|\xi|)^M
\qquad\forall N\in\N,\, \xi\in\R^n;
\eeqs
\item[(b)]
for every $\ell\in\N_0$, $k\in\N$ there exists $C_{k,\ell}>0$ such that
\beqs
\label{B5}
|\widehat{f}_N(\xi)|\leq C_{k,\ell}e^{k\varphi^*(Nm/k)}
(1+|\xi|)^{-\ell}
\qquad\forall N\in\N,\,\xi\in\Gamma.
\eeqs
\end{itemize}
\item[(iii)]
\underline{Roumieu case}:
\begin{itemize}
\item[(a)]
there exist $k\in\N$, $M,C>0$ such that
\beqs
\label{R4}
|\widehat{f}_N(\xi)|\leq C^N(e^{\frac{1}{Nmk}\varphi^*(Nmk)}
+|\xi|)^{Nm}(1+|\xi|)^M
\qquad\forall N\in\N,\, \xi\in\R^n;
\eeqs
\item[(b)]
there exists $k\in\N$ such that
for every $\ell\in\N_0$ there is $C_\ell>0$ s.t.
\beqs
\label{R5}
|\widehat{f}_N(\xi)|\leq C_\ell e^{\frac1k\varphi^*(Nmk)}
(1+|\xi|)^{-\ell}
\qquad\forall N\in\N,\,\xi\in\Gamma.
\eeqs
\end{itemize}
\end{itemize}
\end{Def}

Now we want to characterize the $\omega$-wave front set in terms of condition
\eqref{B1} for the Beurling case, and \eqref{R1} for the Roumieu case. To do this, we first give the following definition of wave front set
and we prove in Theorem \ref{propWF} below its equivalence to
the one of Definition \ref{defWF1}:
\begin{Def}
Let $P(D)$ be a hypoelliptic linear partial differential
operator with constant coefficients; let $\omega$ be a non-quasianalytic
weight function with
$\omega(t^\gamma)=o(\sigma(t))$, as $t\to+\infty$, where $\gamma$ is the
constant defined in \eqref{gamma} and $\sigma(t)=t^{1/s}$ for some
$s>1$; let $\Omega$ be an open subset of $\R^n$ and $u\in\D'(\Omega)$.

We say that $(x_0,\xi_0)\in\Omega\times(\R^n\setminus\{0\})$ is not
in the $(\omega)$-wave front set $\WF_{(\omega),P}(u)$ (resp.
$\{\omega\}$-wave front set $\WF_{\{\omega\},P}(u)$) with respect
to the iterates of $P$, if there exist a neighborhood $U$ of $x_0$,
an open conic neighborhood $\Gamma$ of $\xi_0$ and
$\psi\in\D_{\{\sigma\}}(\Omega)$ with $\psi\equiv1$ in $U$ such that
the following condition $(i)$ (resp. $(ii)$) holds:
\begin{itemize}
\item[(i)]
\underline{Beurling case}:
For every $k\in\N$ there exists $C_k>0$ such that
\beqs
\label{B3}
|\widehat{\psi u}(\xi)|\leq C_ke^{-k\omega(|P(\xi)|^{1/m})}
\qquad\forall\xi\in\Gamma.
\eeqs
\item[(ii)]
\underline{Roumieu case}:
There exist $k\in\N$, $C>0$ such that
\beqs
\label{R3}
|\widehat{\psi u}(\xi)|\leq Ce^{-\frac1k\omega(|P(\xi)|^{1/m})}
\qquad\forall\xi\in\Gamma.
\eeqs
\end{itemize}
\end{Def}

In order to prove that $\WF_*^P(u)=\WF_{*,P}(u)$, for $*=(\omega)$ or
$\{\omega\}$, let's start by the following:

\begin{Lemma}
\label{prop39}
Let $P(D)$ be a hypoelliptic linear partial differential operator of order $m$
with constant coefficients;
let $\omega$ be a non-quasianalytic weight function such that
$\omega(t^\gamma)=o(\sigma(t))$, as $t\to+\infty$, where $\gamma$ is the
constant defined in \eqref{gamma} and $\sigma(t)=t^{1/s}$ for some
$s>1$; let $\Omega$ be an open subset of $\R^n$ and $u\in\D'(\Omega)$.
Then
\beqsn
\WF_*^P(u)\subseteq\WF_{*,P}(u),
\eeqsn
for $*=(\omega)$ or $\{\omega\}$,
\end{Lemma}

\begin{proof}
\underline{Beurling case}.
Let $(x_0,\xi_0)\notin\WF_{(\omega),P}(u)$.
There exist then a neighborhood $U$ of $x_0$, an open conic neighborhood
neighborhood $\Gamma$ of $\xi_0$ and $\psi\in\D_{\{\sigma\}}(\Omega)$
with $\psi\equiv1$ in $U$ satisfying \eqref{B3}.
Setting $f_N=P(D)^N(\psi u)$ we have that $f_N\in\E'(\Omega)$,
$f_N=P(D)^Nu$ in $U$ and
\beqs
\label{B6}
|\widehat{f}_N(\xi)|\leq|P(\xi)|^N|\widehat{\psi u}(\xi)|
\leq C^N(1+|\xi|)^{Nm}(1+|\xi|)^M
\qquad\forall\xi\in\R^n
\eeqs
for some $M>0$, since $\psi u\in\E'(\R^n)$.
Clearly \eqref{B6} implies \eqref{B4}.

To prove \eqref{B5} take $k\in\N$, $\ell>0$ and $\xi\in\Gamma$.
By \eqref{H1} and \eqref{B3} we have
\beqsn
|\xi|^\ell|\widehat{f}_N(\xi)|\leq&&\delta^{-\ell/d}|P(\xi)|^{\ell/d}|P(\xi)|^N
|\widehat{\psi u}(\xi)|\\
\leq&&\delta^{-\ell/d} C_{2k}(|P(\xi)|^{1/m})^{\frac{\ell m}{d}+Nm}
e^{-2k\omega(|P(\xi)|^{1/m})}\\
\leq&&C_{k,\ell}e^{k\varphi^*\left(\frac{Nm}{k}\right)},
\eeqsn
by \cite[Lemma 16(i)]{BJJ} and the convexity of
$\varphi^*$.

This proves \eqref{B5} and hence  $(x_0,\xi_0)\notin\WF_{(\omega)}^P(u)$.

\underline{Roumieu case}. It's similar to the Beurling case.
\end{proof}

We recall, from \cite[Lemma 4]{FGJ} (see also \cite[Proposition 3.4.4]{R}),
the following lemma that we
shall need later:
\begin{Lemma}
\label{lemma4FGJ}
Let $\Gamma$ and $\Gamma'$ be two cones in $\R^n$ such that
$\Gamma'\subset\subset\Gamma$ in the sense that $\Gamma'\cap S^{n-1}
\subset\subset\Gamma\cap S^{n-1}$, where $S^{n-1}$ is the unit sphere
in $\R^n$.

Then there exists a bounded $\phi\in\E_{(\omega)}(\R^n)\subset
\E_{\{\omega\}}(\R^n)$ with $\supp\phi\subset\Gamma$, $\phi\equiv1$ on
$\Gamma'$ (for large $|\xi|$), which is the symbol of a pseudo-differential
operator $\phi(D)$ satisfying
\beqs
\label{phihat}
\widehat{\phi(D)u}(\xi)=\phi(\xi)\widehat{u}(\xi)\qquad
 u\in\D_{(\omega)}(\R^n),\ \xi\in\R^n.
\eeqs
\end{Lemma}

\begin{Rem}
\begin{em}
Here, the definition of pseudodifferential operator is as in \cite[Def. 3]{FGJ}. Then, we must consider the symbol of the operator $\phi(D)$ as $(2\pi)^{-n} \phi(\xi)$ (compare with the beginning of the proof of \cite[Theorem 2]{FGJ}).
\end{em}
\end{Rem}

We can now prove the following result:
\begin{Th}
\label{thm1}
Let $P(D)$ be a hypoelliptic linear partial differential operator of
order $m$ with constant coefficients; let $\omega$ be a non-quasianalytic
weight function such that $\omega(t^\gamma)=o(\sigma(t))$, as $t\to+\infty$,
where $\gamma$ is the constant defined in \eqref{gamma} and $\sigma(t)=t^{1/s}$
for some $s>1$. Let $\Omega$ be an open subset of $\R^n$,
$u\in\E'(\Omega)$ and $(x_0,\xi_0)\in\Omega\times(\R^n\setminus\{0\})$.

Then $(x_0,\xi_0)\notin\WF_*^P(u)$ if and only if there exists
$\phi$ as in Lemma \ref{lemma4FGJ} such that $\left.\phi(D)u\right|_V
\in\E_*^P(V)$ for some neighborhood $V$ of $x_0$, where $*=(\omega)$
or $\{\omega\}$.
\end{Th}

\begin{proof}
\underline{Beurling case}.
Let $(x_0,\xi_0)\notin\WF_{(\omega)}^P(u)$. There
exist a neighborhood
$U$ of $x_0$, an open conic neighborhood $\Gamma$ of $\xi_0$ and a sequence
$\{f_N\}_{N\in\N}\subset\E'(\Omega)$ with $f_N=P(D)^Nu$ in $U$ and
satisfying \eqref{B4} and \eqref{B5}.

Now, we consider a conic neigborhood $\Gamma'$ of $\xi_0$ with
$\Gamma'\subset\subset\Gamma$.
Take then $\phi$ as in Lemma \ref{lemma4FGJ} and define
$h_N:=\phi(D)f_N$.
Since $\phi(D)$ and $P(D)^N$ commute, we have that
$h_N=P(D)^N\phi(D)u$ in a neighborhood of $x_0$ and
$\widehat{h}_N(\xi)=\phi(\xi)\widehat{f}_N(\xi)$ satisfies
\eqref{B0} because of \eqref{B5} and $\phi$ bounded with $\supp\phi\subset
\Gamma$.
Therefore $\left.\phi(D)u\right|_V\in\E_{(\omega)}^P(V)$
for a neighborhood $V$ of $x_0$ by Theorem \ref{thB}.

Vice versa, let us now assume $\left.\phi(D)u\right|_V\in\E_{(\omega)}^P(V)$,
where $\phi$ is
as in Lemma \ref{lemma4FGJ}, for some neighborhood $V$ of $x_0$ and
for some neighborhoods $\Gamma'\subset\subset\Gamma$ of $\xi_0$.
Take $\varphi\in\D_{\{\sigma\}}(V)$ with $\varphi\equiv1$ in a
neighborhood of $x_0$ and write
\beqsn
\varphi u=\varphi\phi(D)u+\varphi(I-\phi(D))u=:u_1+u_2.
\eeqsn
Since $\left.\phi(D)u\right|_V\in\E_{(\omega)}^P(V)$ by assumption and
$\varphi\in\D_{\{\sigma\}}(V)$, we can apply \cite[Prop. 17]{BJJ} to get
$u_1\in\E_{(\omega)}^P(V)$. We thus have to prove that
$(x_0,\xi_0)\notin\WF_{(\omega)}^P(u_2)$. By Lemma \ref{prop39} it's
enough to prove that $(x_0,\xi_0)\notin\WF_{(\omega),P}(u_2)$. To do this, we have to find $\psi\in\D_{\{\sigma\}}(\Omega)$ with $\psi\equiv1$
in a neighborhood of $x_0$ and $\widehat{\psi u_2}(\xi)$
satisfying \eqref{B3} in a conic neighborhood of $\xi_0$.

Taking $\psi\in\D_{\{\sigma\}}(\Omega)$ with $\psi\equiv1$ on $\supp\varphi$
we have that $\psi u_2=u_2$ by the definition of $u_2$. Hence we will
prove \eqref{B3} for $\widehat{u}_2$. We have
\beqs
\nonumber
\widehat{u}_2=&&\F(\varphi(I-\phi(D))u)(\xi)
=(2\pi)^{-n}(\widehat{\varphi}\ast(1-\phi)\widehat{u})(\xi)\\
\label{u21}
=&&(2\pi)^{-n}\int\widehat{\varphi}(\xi-\eta)(1-\phi)(\eta)\widehat{u}(\eta)d\eta.
\eeqs
Let $\Lambda$ be a conic neigborhood of $\xi_0$ with $\bar{\Lambda}
\subset\Gamma'$, so that $(1-\phi)(\eta)=0$ in $\Lambda$.
For fixed $\xi\in\Lambda$ we set
\beqsn
&&A:=\{\eta\in\R^n:\ |\xi-\eta|\leq\delta(|\xi|+|\eta|)\}\\
&&B:=\{\eta\in\R^n:\ |\xi-\eta|>\delta(|\xi|+|\eta|)\}
\eeqsn
for $\delta>0$ small enough so that $A\subset\Gamma'$ and hence
$(1-\phi)(\eta)=0$ in $A$.

Splitting  the integral \eqref{u21} we have, for $\xi\in\Lambda$:
\beqs
\nonumber
\widehat{u}_2(\xi)=&&(2\pi)^{-n}\int_A\widehat{\varphi}(\xi-\eta)(1-\phi)(\eta)
\widehat{u}(\eta)d\eta+(2\pi)^{-n}
\int_B\widehat{\varphi}(\xi-\eta)(1-\phi)(\eta)\widehat{u}(\eta)d\eta\\
\label{161}
=&&(2\pi)^{-n}\int_B\widehat{\varphi}(\xi-\eta)(1-\phi)(\eta)\widehat{u}(\eta)
d\eta.
\eeqs
For $\xi\in\Lambda$ and $\eta\in B$, as
$\varphi\in\D_{\{\sigma\}}(\Omega)$ and
$\sigma(t)=t^{1/s}$, there exist
$C,C',\varepsilon,\varepsilon'>0$ such that
\beqsn
|\widehat{\varphi}(\xi-\eta)|\leq C e^{-\varepsilon\sigma(|\xi-\eta|)}
\leq C'e^{-\varepsilon'(\sigma(|\xi|)+\sigma(|\eta|))}.
\eeqsn
Moreover, $|\widehat{u}(\eta)|\leq C(1+|\eta|)^M$ for some $C,M>0$ since
$u\in\E'(\Omega)$, and therefore
\beqs
\nonumber
\left|\int_B\widehat{\varphi}(\xi-\eta)(1-\phi)(\eta)\widehat{u}(\eta)d\eta
\right|
\leq&& C'e^{-\varepsilon'(\sigma(\xi))}\int|1-\phi(\eta)|
|\widehat{u}(\eta)|e^{-\varepsilon'(\sigma(\eta))}d\eta\\
\nonumber
\leq&&De^{-\varepsilon'(\sigma(\xi))}\int(1+|\eta|)^Me^{-\varepsilon'(\sigma(\eta))}
d\eta\\
\label{162}
\leq&& D'e^{-\varepsilon'(\sigma(\xi))},
\eeqs
for some $D,D'>0$, because $1-\phi(\eta)$ is bounded.

Since $\omega(t^\gamma)=o(\sigma(t))$ with $\gamma\geq1$ (see Remark
\ref{rem22}), for every $k\in\N$ there exists $R_k>0$ such that
for all $|\xi|\geq R_k$:
\beqs
\label{163}
\omega(|P(\xi)|^{\frac1m})\leq\omega(c|\xi|)\leq L(\omega(\xi)+1)
\leq\frac1k\sigma(\xi)+L
\eeqs
for some $c,L>0$, and hence, from \eqref{162} and \eqref{161}, for
every $k\in\N$ there is $D_k>0$ such that:
\beqsn
|\widehat{u}_2(\xi)|\leq D_ke^{-k\omega(|P(\xi)|^{1/m})}\qquad\forall\xi\in\Lambda,
\eeqsn
proving that $(x_0,\xi_0)\notin\WF_{(\omega),P}(u_2)$ and hence
$(x_0,\xi_0)\notin\WF_{(\omega)}^P(u)$.

\underline{Roumieu case}.
The proof is analogous to that of the Beurling case.
If $(x_0,\xi_0)\notin\WF_{\{\omega\}}^P(u)$ we find a neighborhood
$U$ of $x_0$, an open conic neighborhood $\Gamma$ of $\xi_0$ and a sequence
$\{f_N\}_{N\in\N}\subset\E'(\Omega)$ satisfying $(i)$, $(iii)(a)$
and $(iii)(b)$ of Definition \ref{defWF1}.
As in the Beurling case $h_N:=\phi(D)f_N$ satisfies the desired estimate
\eqref{R0}, so that $\left.\phi(D)u\right|_V\in
\E_{\{\omega\}}^P(V)$ for a neighborhood $V$ of $x_0$ by Theorem \ref{thR}.

Conversely, if $\left.\phi(D)u\right|_V\in
\E_{\{\omega\}}^P(V)$ we proceed as in the Beurling case and obtain that
$u_1:=\varphi\phi(D)u\in\E_{\{\omega\}}^P(V)$ by Proposition 17 of \cite{BJJ}.
The proof that $(x_0,\xi_0)\notin\WF_{\{\omega\}}^P(u_2)$ is analogous to
the one in the Beurling case, applying \eqref{163} for $k=1$.
\end{proof}

\begin{Prop}
\label{lemma2'}
Let $P(D)$ be a hypoelliptic linear partial differential operator
with constant coefficients;
let $\omega$ be a non-quasianalytic weight function such that
$\omega(t^\gamma)=o(\sigma(t))$, as $t\to+\infty$, where $\gamma$ is the
constant defined in \eqref{gamma} and $\sigma(t)=t^{1/s}$ for some
$s>1$; let $\Omega$ be an open subset of $\R^n$ and $u\in\D'(\Omega)$.
If $\psi\in\D_{\{\sigma\}}(\Omega)$, then
\beqsn
\WF_*^P(\psi u)\subseteq\WF_*^P(u),
\eeqsn
for $*=(\omega)$ or $\{\omega\}$.
\end{Prop}

\begin{proof}

\underline{Beurling case}.
Let $(x_0,\xi_0)\notin\WF_{(\omega)}^P(u)$. First, we observe that we can assume
$u\in\E'(\Omega)$, since the definition of the wave front set is local. Then, there exist
a neighborhood
$V$ of $x_0$ and a conic neighborhood $\Gamma$ of $\xi_0$ such that
$$(V\times\Gamma)\cap\WF_{(\omega)}^P(u)=\emptyset.$$

For $\Gamma'\subset\subset\Gamma$ there exists, by Theorem \ref{thm1},
a bounded $\phi\in\E_{(\omega)}(\R^n)$ with $\supp\phi\subset\Gamma$,
$\phi\equiv1$ in $\Gamma'$ such that $\left.\phi(D)u\right|_{V'}\in
\E_{(\omega)}^P(V')$ for some neighborhood $V'\subseteq V$ of $x_0$.

Let us then consider $\varphi\in\D_{\{\sigma\}}(V')$ with $\varphi\equiv1$ in
a neighborhood of $x_0$, and set
\beqsn
\varphi\psi u=\varphi\psi\phi(D)u+\varphi\psi(I-\phi(D))u=:u_1+u_2.
\eeqsn
Since $\varphi,\psi\in\D_{\{\sigma\}}(\Omega)$ we have, by
\cite[Prop. 4.4]{BMT}, that $\varphi\psi\in\D_{\{\sigma\}}(\Omega)$ and hence
$u_1\in\E_{(\omega)}^P(V')$ by \cite[Prop. 17]{BJJ} because
$\left.\phi(D)u\right|_{V'}\in\E_{(\omega)}^P(V')$.

Arguing as in the proof of Theorem \ref{thm1} with $\varphi\psi\in
\D_{\{\sigma\}}$ instead of $\varphi\in\D_{\{\sigma\}}$, we obtain that
$(x_0,\xi_0)\notin\WF_{(\omega),P}(u_2)$ and hence
$(x_0,\xi_0)\notin\WF_{(\omega)}^P(u_2)$ by Lemma \ref{prop39}.

Therefore $(x_0,\xi_0)\notin\WF_{(\omega)}^P(\psi u)$.

\underline{Roumieu case}.
It is similar to the Beurling case.
\end{proof}

\begin{Th}
\label{propWF}
Let $P(D)$ be a hypoelliptic linear partial differential operator of order $m$
with constant coefficients;
let $\omega$ be a non-quasianalytic weight function such that
$\omega(t^\gamma)=o(\sigma(t))$, as $t\to+\infty$, where $\gamma$ is the
constant defined in \eqref{gamma} and $\sigma(t)=t^{1/s}$ for some
$s>1$; let $\Omega$ be an open subset of $\R^n$ and $u\in\D'(\Omega)$.
Then
\beqsn
\WF_*^P(u)=\WF_{*,P}(u),
\eeqsn
for $*=(\omega)$ or $\{\omega\}$,
\end{Th}

\begin{proof}
The inclusion
\beqsn
\WF_*^P(u)\subseteq\WF_{*,P}(u)
\eeqsn
has been proved in Lemma \ref{prop39}.

Let us prove the other inclusion.

\underline{Beurling case}.
Let $(x_0,\xi_0)\notin\WF_{(\omega)}^P(u)$. There exist a compact neighborhood
$K$ of $x_0$ and a closed conic neighborhood $F$ of $\xi_0$ such that
\beqsn
\WF_{(\omega)}^P(u)\cap(K\times F)=\emptyset.
\eeqsn
Take, according to \cite[Lemma 2.2]{Huniq}, $\chi_N\subset
\D(K)$ with $\chi_N\equiv1$ in a neighborhood $K'\subset K$ of $x_0$, that
satisfies
\beqs
\label{B9}
\sup_K|D^{\alpha+\beta}\chi_N|\leq C_\alpha(C_\alpha N)^{|\beta|}
\qquad\forall\alpha,\beta\in\N_0,\,|\beta|\leq N.
\eeqs
Fix  $\psi\in\D_{\{\sigma\}}(K')$. By Proposition \ref{lemma2'}
\beqsn
\WF_{(\omega)}^P(\psi u)\subseteq\WF_{(\omega)}^P(u),
\eeqsn
and, hence,
\beqsn
\WF_{(\omega)}^P(\psi u)\cap(K\times F)=\emptyset.
\eeqsn
Now, consider $g_N:=\chi_{Nmp}P(D)^N\psi u$, for $p$ sufficiently large so that,
by \cite[Corollary 10]{BJJ}, for every $k,\ell\in\N$ there exists $C_{k,\ell}>0$
such that
\beqs
\label{B8}
|\widehat{g}_N(\xi)|\leq C_{k,\ell}e^{k\varphi^*\left(\frac{Nm}{k}\right)}
(1+|\xi|)^{-\ell}
\qquad\forall N\in\N,\,\xi\in F.
\eeqs
Moreover,
\beqsn
g_N=\chi_{Nmp}P(D)^N\psi u=P(D)^N\psi u\qquad\mbox{in}\  \R^n,
\eeqsn
since $\chi_{Nmp}\equiv1$ on $\supp\psi$.

Therefore
\beqsn
|P(\xi)^N\widehat{\psi u}(\xi)|=|\F(P(D)^N\psi u)(\xi)|
=|\widehat{g}_N(\xi)|
\leq C_{k,\ell}e^{k\varphi^*\left(\frac{Nm}{k}\right)}(1+|\xi|)^{-\ell}
\qquad\forall N\in\N,\ \xi\in F.
\eeqsn
By Lemma \ref{lemmaBR}, this implies that
\beqsn
|\widehat{\psi u}(\xi)|\leq C_ke^{-k\omega(|P(\xi)|^{1/m})}
\qquad\forall\xi\in F,
\eeqsn
i.e. $(x_0,\xi_0)\notin\WF_{(\omega),P}(u)$.

\underline{Roumieu case}. It is similar to the Beurling case.
\end{proof}

By Theorem \ref{propWF} and Lemma \ref{lemmaBR}
we can now prove the following characterizations of the wave front
set with respect to the iterates:
\begin{Cor}
\label{corWFB}
Let $P(D)$ be a hypoelliptic linear partial differential operator
with constant coefficients;
let $\omega$ be a non-quasianalytic weight function such that
$\omega(t^\gamma)=o(\sigma(t))$, as $t\to+\infty$, where $\gamma$ is the
constant defined in \eqref{gamma} and $\sigma(t)=t^{1/s}$ for some
$s>1$; let $\Omega$ be an open subset of $\R^n$, $u\in\D'(\Omega)$
and $(x_0,\xi_0)\in\Omega\times(\R^n\setminus\{0\})$.

The following conditions are equivalent:
\begin{itemize}
\item[(1)]
$(x_0,\xi_0)\notin\WF_{(\omega)}^P(u)$.
\item[(2)]
There exist a neighborhood $U$ of $x_0$, an open conic neighborhood
$\Gamma$ of $\xi_0$ and $\psi\in\D_{\{\sigma\}}(\Omega)$ with
$\psi\equiv1$ in $U$ such that:
\beqs
\nonumber
&&\forall k\in\N,\ \exists C_k>0:\\
\label{BB3}
&&|\widehat{\psi u}(\xi)|\leq C_ke^{-k\omega(|P(\xi)|^{1/m})}
\qquad\forall\xi\in\Gamma.
\eeqs
\item[(3)]
There exist a neighborhood $U$ of $x_0$, an open conic neighborhood
$\Gamma$ of $\xi_0$ and $\psi\in\D_{\{\sigma\}}(\Omega)$ with
$\psi\equiv1$ in $U$ such that:
\beqs
\nonumber
&&\forall k,\ell\in\N,\ \exists C_{k,\ell}>0:\\
\label{BBB3}
&&|P(\xi)|^N|\widehat{\psi u}(\xi)|\leq C_{k,\ell}
e^{k\varphi^*\left(\frac{Nm}{k}\right)}(1+|\xi|)^{-\ell}
\qquad\forall N\in\N,\,\xi\in\Gamma.
\eeqs
\item[(4)]
There exist a neighborhood $U$ of $x_0$, an open conic neighborhood
$\Gamma$ of $\xi_0$ and a bounded sequence $\{u_N\}_{N\in\N}
\subset\E'(\Omega)$ such that $u_N=u$ in $U$ and:
\beqs
\nonumber
&&\forall k\in\N,\ell\in\N_0,\ \exists C_{k,\ell}>0:\\
\label{C1}
&&|P(\xi)|^N|\widehat{u}_N(\xi)|\leq C_{k,\ell}
e^{k\varphi^*\left(\frac{Nm}{k}\right)}(1+|\xi|)^{-\ell}
\qquad\forall N\in\N,\,\xi\in\Gamma.
\eeqs
\end{itemize}
\end{Cor}
\begin{Cor}
\label{corWFR}
Let $P(D)$ be a hypoelliptic linear partial differential operator
with constant coefficients;
let $\omega$ be a non-quasianalytic weight function such that
$\omega(t^\gamma)=o(\sigma(t))$, as $t\to+\infty$, where $\gamma$ is the
constant defined in \eqref{gamma} and $\sigma(t)=t^{1/s}$ for some
$s>1$; let $\Omega$ be an open subset of $\R^n$, $u\in\D'(\Omega)$
and $(x_0,\xi_0)\in\Omega\times(\R^n\setminus\{0\})$.

The following conditions are equivalent:
\begin{itemize}
\item[(1)]
$(x_0,\xi_0)\notin\WF_{\{\omega\}}^P(u)$.
\item[(2)]
There exist a neighborhood $U$ of $x_0$, an open conic neighborhood
$\Gamma$ of $\xi_0$ and $\psi\in\D_{\{\sigma\}}(\Omega)$ with
$\psi\equiv1$ in $U$ such that:
\beqs
\nonumber
&&\exists k\in\N, C>0:\\
\label{RR3}
&&|\widehat{\psi u}(\xi)|\leq Ce^{-\frac1k\omega(|P(\xi)|^{1/m})}
\qquad\forall\xi\in\Gamma.
\eeqs
\item[(3)]
There exist a neighborhood $U$ of $x_0$, an open conic neighborhood
$\Gamma$ of $\xi_0$ and $\psi\in\D_{\{\sigma\}}(\Omega)$ with
$\psi\equiv1$ in $U$ such that:
\beqs
\nonumber
&&\exists k\in\N,\forall\ell\in\N,\ \exists C_\ell>0:\\
\label{RRR3}
&&|P(\xi)|^N|\widehat{\psi u}(\xi)|\leq C_\ell
e^{\frac1k\varphi^*(Nmk)}(1+|\xi|)^{-\ell}
\qquad\forall N\in\N,\,\xi\in\Gamma.
\eeqs
\item[(4)]
There exist a neighborhood $U$ of $x_0$, an open conic neighborhood
$\Gamma$ of $\xi_0$ and a bounded sequence $\{u_N\}_{N\in\N}
\subset\E'(\Omega)$ such that $u_N=u$ in $U$ and:
\beqs
\nonumber
&&\exists k\in\N,\forall\ell\in\N_0,\ \exists C_\ell>0:\\
\label{C2}
&&|P(\xi)|^N|\widehat{u}_N(\xi)|\leq C_\ell
e^{\frac1k\varphi^*(Nmk)}(1+|\xi|)^{-\ell}
\qquad\forall N\in\N,\,\xi\in\Gamma.
\eeqs
\end{itemize}
\end{Cor}

\begin{proof}[Proof of Corollary \ref{corWFB}]

$(1)\Leftrightarrow(2)$ follows from Theorem \ref{propWF}.

$(2)\Leftrightarrow(3)$ follows from Lemma \ref{lemmaBR}.

$(3)\Rightarrow(4)$: Taking $u_N=\psi u$ we have that
$\{u_N\}_{N\in\N}$ is a bounded sequence in $\E'(\Omega)$, $u_N=u$ in $U$
 and \eqref{BBB3} implies \eqref{C1} by the choice of $u_N$.

$(4)\Rightarrow(1)$: Taking $f_N=P(D)^Nu_N$ we have that $f_N=P(D)^Nu$
in $U$ and, by \eqref{C1},
\beqsn
|\widehat{f}_N(\xi)|=|P(\xi)|^N|\widehat{u}_N(\xi)|
\leq C_{k,\ell}e^{k\varphi^*\left(\frac{Nm}{k}\right)}(1+|\xi|)^{-\ell}
\qquad\forall\xi\in\Gamma,
\eeqsn
giving condition \eqref{B5}.

Finally, since $\{u_N\}_{N\in\N}$ is a bounded sequence in $\E'(\Omega)$,
there exist $c,M>0$ such that
\beqsn
|\widehat{u}_N(\xi)|\leq c(1+|\xi|)^M\qquad
\forall\xi\in\R^n
\eeqsn
and hence
\beqsn
|\widehat{f}_N(\xi)|\leq&&|P(\xi)|^N|\widehat{u}_N(\xi)|
\leq C^N(1+|\xi|)^{mN}c(1+|\xi|)^M\\
\leq&& \tilde{C}^N(e^{\frac{k}{Nm}\varphi^*\left(\frac{Nm}{k}\right)}+|\xi|)^{Nm}
(1+|\xi|)^M
\qquad\forall\xi\in\R^n,
\eeqsn
for some $C,\tilde{C}>0$,
proving also \eqref{B4}.

Therefore $(x_0,\xi_0)\notin\WF_{(\omega)}^P(u)$.
\end{proof}

\begin{proof}[Proof of Corollary \ref{corWFR}]
Its similar to the Beurling case, Corollary \ref{corWFB}.
\end{proof}

The new characterization of $\WF_*^P(u)$ given by
Corollaries \ref{corWFB} and \ref{corWFR} allows to complete
Theorem 18 of \cite{BJJ}, obtaining the existence of a distribution with
prescribed $\omega$-wave front set with respect to the iterates:
\begin{Th}
\label{newth43}
Let $P(D)$ be a linear partial differential operator of order $m$ with
constant coefficients which is hypoelliptic, but not elliptic.
Let $\omega$ be a non-quasi-analytic weight function such that
$\omega(t^b)=o(\bar{\sigma}(t))$, as $t\to+\infty$, where
$\bar{\sigma}(t)=t^{1/s}$ for some $s>1$ and $b=\max\{\gamma,3/2\}$, with
$\gamma$ defined in \eqref{gamma}.

Given an open subset $\Omega$ of $\R^n$ and a closed conic subset
$S$ of $\Omega\times(\R^n\setminus\{0\})$, there exists $u\in\D'(\Omega)$ with
\beqsn
\WF_*^P(u)=S,
\eeqsn
for $*=(\omega)$ or $\{\omega\}$.
\end{Th}

\begin{proof}
We assume, without loss of generality, that $\Omega=\R^n$; we
construct the same distribution $u\in\D'(\R^n)$ constructed
in \cite[Thm. 18]{BJJ} and follow the ideas therein and in
\cite[Thm. 8.1.4]{H1}.

We choose a sequence $(x_k,\theta_k)\in S$ with $|\theta_k|=1$ so that
every $(x,\theta)\in S$ with $|\theta|=1$ is the limit of a subsequence.

We set $\sigma(t):=\omega(t^{3/2})$ and separate the Beurling and
the Roumieu cases.

\underline{Beurling case}:
Take $\phi\in\D_{(\omega)}(\R^n)$ with $\widehat{\phi}(0)=1$ and define, as in
(158) of \cite{BJJ},
\beqs
\label{442}
u(x)=\sum_{k=1}^{+\infty}e^{-\sigma(k^{d/m})}\phi(k(x-x_k))
e^{ik^3\langle x,\theta_k\rangle},
\eeqs
where $d$ is the constant of \eqref{H1} with $0<d<m$ since $P$ is not elliptic
by assumption.
This is a continuous function in $\R^n$ and it was already proved in
\cite[Thm. 18]{BJJ} that
\beqsn
\emptyset\not=\WF_{(\omega)}^P(u)\subseteq S.
\eeqsn

Let us now prove the other inclusion.
Fix $(x_0,\xi_0)\in S$ and assume by contradiction that
$(x_0,\xi_0)\notin\WF_{(\omega)}^P(u)$.
Then, by Corollary \ref{corWFB}, there exist a neighborhood $U$ of $x_0$,
an open conic neighborhood $\Gamma$ of $\xi_0$ and
$\psi\in\D_{\{\bar{\sigma}\}}(\R^n)$ with $\psi\equiv1$ in $U$ such that
\eqref{BBB3} is satisfied.

Set then
\beqsn\phi_k(y)=\psi\left(\frac yk +x_k\right)\phi(y).
\eeqsn

Since $\sigma(t)=o(\bar{\sigma}(t))$ by assumption, we have that
$\D_{\{\bar{\sigma}\}}(\R^n)\subseteq\D_{(\sigma)}(\R^n)$ by
\cite[Prop. 4.7]{BMT} and therefore $\{\phi_k\}_{k\in\N}$ is a bounded sequence
in $\D_{(\sigma)}(\R^n)$, taking into account that
$\supp\phi_k\subseteq\supp\phi$ for all $k\in\N$.
By \cite[Prop. 3.4]{BMT}, for each $h\in\N$ there is  $C_h>0$ such that
\beqs
\label{phij}
|\widehat{\phi}_j(\xi)|\leq C_h e^{-h\sigma(\xi)}
\qquad\forall j\in\N,\,\xi\in\R^n.
\eeqs
Moreover, following \cite{BJJ}, we have
\beqsn
\F(P(D)^N(\psi u))(\xi)=&&
\sum_{j=1}^{+\infty}e^{-\sigma(j^{d/m})}P(\xi)^N
\F\left(\phi_j(j(x-x_j))e^{ij^3\langle x,\theta_j\rangle}\right)\\
=&&\sum_{j=1}^{+\infty}e^{-\sigma(j^{d/m})}P(\xi)^N
j^{-n}\widehat{\phi}_j\left(\frac{\xi-j^3\theta_j}{j}\right)
e^{i\langle x_j,j^3\theta_j-\xi\rangle}.
\eeqsn
If $x_k$ is close to $x_0$ and $k$ is large enough, then $\phi_k=\phi$ and
\beqsn
|\F(P(D)^N(\psi u))(k^3\theta_k)|=&&\bigg|e^{-\sigma(k^{d/m})}k^{-n}
P(k^3\theta_k)^N\\
&&+\sum_{j\not=k}e^{-\sigma(j^{d/m})}j^{-n}P(k^3\theta_k)^N
\widehat{\phi}_j\left(\frac{k^3\theta_k-j^3\theta_j}{j}\right)
e^{i\langle x_j,j^3\theta_j-k^3\theta_k\rangle}\bigg|\\
\geq&&|P(k^3\theta_k)|^N\left(e^{-\sigma(k^{d/m})}k^{-n}
-\sum_{j\not=k}e^{-\sigma(j^{d/m})}j^{-n}C_h
e^{-h\sigma\left(\frac{k^3\theta_k-j^3\theta_j}{j}\right)}\right)\\
\geq&&\delta^Nk^{3Nd}(e^{-\sigma(k^{d/m})}k^{-n}-C'_he^{-h\sigma(k)}),
\eeqsn
for some $C'_h>0$,
because of \eqref{phij}, \eqref{H1} and
\beqsn
|k^3\theta_k-j^3\theta_j|\geq|k^3-j^3|\geq k^2+kj+j^2\geq kj
\qquad\mbox{if} \ k\not=j.
\eeqsn
But for every fixed $h\geq2$ there exists $k_0\in\N$ such that
\beqsn
\frac{\log C'_h}{\sigma(k)}+\frac{\log2}{\sigma(k)}
+\frac{n\log k}{\sigma(k)}+\frac{\sigma(k^{d/m})}{\sigma(k)}\leq h
\qquad\forall k\geq k_0
\eeqsn
since $\log k=o(\sigma(k))$ and $0<d<m$.

Therefore
\beqs
\label{444}
|\F(P(D)^N(\psi u))(k^3\theta_k)|\geq\frac12\delta^N k^{3Nd}k^{-n}
e^{-\sigma(k^{d/m})}
\qquad\forall k\geq k_0.
\eeqs

On the other hand, by \eqref{BBB3}:
\beqs
\label{445}
|\F(P(D)^N(\psi u))(k^3\theta_k)|=|P(k^3\theta_k)|^N
|\widehat{\psi u}(k^3\theta_k)|
\leq C_{h,\ell}e^{h\varphi^*\left(\frac{Nm}{h}\right)}
(1+|k^3\theta_k|)^{-\ell}.
\eeqs

But \eqref{444} and \eqref{445} give a contradiction for $k$ large enough
(see \cite{BJJ} for more details).
Therefore $(x_0,\xi_0)\in\WF_{(\omega)}^P(u)$ and
$S\subseteq\WF_{(\omega)}^P(u)$.

\underline{Roumieu case}:
Take $\phi\in\D_{\{\sigma\}}(\R^n)$ with $\widehat{\phi}(0)=1$; choose, by
Lemma 1.7 of \cite{BMT}, a non-quasianalytic
weight function $\alpha(t)$ such that
$\log t=o(\alpha(t))$ and $\alpha(t)=o(\sigma(t))$ for $t\to+\infty$,
and define, as in (138) of \cite{BJJ},
\beqsn
u(x)=\sum_{k=1}^{+\infty}e^{-\frac{\sigma(k^{d/m})}{\alpha(k^{d/m})}\log k}
\phi(k(x-x_k))
e^{ik^3\langle x,\theta_k\rangle}.
\eeqsn
This is a continuous function in $\R^n$ and it was already proved in
\cite[Thm. 18]{BJJ} that
\beqsn
\emptyset\not=\WF_{\{\omega\}}^P(u)\subseteq S.
\eeqsn
Let us now prove the other inclusion.
Fix $(x_0,\xi_0)\in S$ and assume by contradiction that
$(x_0,\xi_0)\notin\WF_{\{\omega\}}^P(u)$.
Then, by Corollary \ref{corWFR}, there exist a neighborhood $U$ of $x_0$,
an open conic neighborhood $\Gamma$ of $\xi_0$ and
$\psi\in\D_{\{\bar{\sigma}\}}(\R^n)$ with $\psi\equiv1$ in $U$ such that
\eqref{RRR3} is satisfied.

Define, as in the Beurling case,
$\phi_k(y)=\psi\left(\frac yk +x_k\right)\phi(y)\in\D_{\{\sigma\}}(\R^n)$
such that, by \cite[Prop. 3.4]{BMT}, there exists $C,h>0$:
\beqsn
|\widehat{\phi}_j(\xi)|\leq C e^{-\frac1h\sigma(\xi)}
\qquad\forall j\in\N,\,\xi\in\R^n.
\eeqsn
Then, if $x_k$ is close to $x_0$ and $k$ is large enough,
\beqs
\nonumber
|\F(P(D)^N(\psi u))(k^3\theta_k)|=&&
\bigg|e^{-\frac{\sigma(k^{d/m})}{\alpha(k^{d/m})}\log k}k^{-n}
P(k^3\theta_k)^N\\
\nonumber
&&+\sum_{j\not=k}e^{-\frac{\sigma(j^{d/m})}{\alpha(j^{d/m})}\log j}
j^{-n}P(k^3\theta_k)^N
\widehat{\phi}_j\left(\frac{k^3\theta_k-j^3\theta_j}{j}\right)
e^{i\langle x_j,j^3\theta_j-k^3\theta_k\rangle}\bigg|\\
\nonumber
\geq&&|P(k^3\theta_k)^N|\bigg(e^{-\frac{\sigma(k^{d/m})}{\alpha(k^{d/m})}\log k}
k^{-n}\\
\nonumber
&&-\sum_{j\not=k}e^{-\frac{\sigma(j^{d/m})}{\alpha(j^{d/m})}\log j}
j^{-n}C
e^{-\frac1h\sigma\left(\frac{k^3\theta_k-j^3\theta_j}{j}\right)}\bigg)\\
\nonumber
\geq&&\delta^Nk^{3Nd}(e^{-\frac{\sigma(k^{d/m})}{\alpha(k^{d/m})}\log k}
k^{-n}-C'e^{-\frac1h\sigma(k)})\\
\label{437}
\geq&&\frac12\delta^N k^{3Nd}
e^{-\frac{\sigma(k^{d/m})}{\alpha(k^{d/m})}\log k}k^{-n}
\eeqs
if $k$ is large enough, because
\beqsn
\frac{\log(2C')}{\sigma(k)}
+\frac{\sigma(k^{d/m})}{\sigma(k)}\frac{\log k}{\alpha(k^{d/m})}
+\frac{n\log k}{\sigma(k)}\leq\frac1h
\eeqsn
if $k$ is sufficiently large, since $\log k=o(\alpha(k))$,
$\log k=o(\sigma(k))$ and $0<d<m$.

On the other hand, by \eqref{RRR3} there exists $h\in\N$ such that
\beqsn
|\F(P(D)^N(\psi u))(k^3\theta_k)|=|P(k^3\theta_k)|^N
|\widehat{\psi u}(k^3\theta_k)|
\leq C_\ell e^{\frac1h\varphi^*(Nmh)}(1+|k^3\theta_k|)^{-\ell}
\eeqsn
which contradicts \eqref{437} (see \cite{BJJ} for more details).

Therefore $(x_0,\xi_0)\in\WF_{\{\omega\}}^P(u)$ and
$\WF_{\{\omega\}}^P(u)=S$.
\end{proof}

\begin{Ex}
\begin{em}
Let $\Omega$ be an open subset of $\R^n$ and
$P(D)$ a linear partial differential operator with constant
coefficients and of order $m$. By Theorems 4.1 and 4.8 of \cite{AJO}, for $*=(\omega)$ or $\{\omega\}$, we have
\beqs
\label{remdavid1}
\WF_*(u)\subseteq\WF_*(Pu)\cup\Sigma,\qquad u\in\D'(\Omega),
\eeqs
where
\beqsn
\Sigma:=\{(x,\xi)\in\Omega\times\R^n\setminus\{0\}:\ P_m(\xi)=0\}
\eeqsn
is the characteristic set of $P$ (here $P_m$ is the principal
part of $P$).

This implies that if $u$ is a solution of $P(D)u=f$,
then
\beqs
\label{remdavid2}
\WF_*(u)\subseteq\Sigma, \qquad\mbox{for each}\ f\in\E_*(\Omega),
\eeqs
since $\WF_*(f)=\emptyset$ for $f\in\E_*(\Omega)$.

Now, by Theorem 13 of \cite{BJJ}
\beqsn
\WF_*(u)\subseteq\WF_*^P(u)\cup\Sigma,\qquad u\in\D'(\Omega),
\eeqsn
so that we can improve \eqref{remdavid2} saying that
\beqsn
\WF_*(u)\subseteq\Sigma \qquad\mbox{if}\ f\in\E_*^P(\Omega).
\eeqsn
Indeed, if $f\in\E_*^P(\Omega)\setminus\E_*(\Omega)$ then
$\WF_*^P(f)=\emptyset$ and hence
\beqsn
\WF_*(u)\subseteq\WF_*^P(u)\cup\Sigma
=\WF_*^P(Pu)\cup\Sigma=\WF_*^P(f)\cup\Sigma
=\Sigma.
\eeqsn
\end{em}
\end{Ex}

\vspace{3mm}
Let us now prove a result that establishes the
relationship between the wave front set in the Beurling class and in
the Roumieu class. See also \cite[Proposition 4.5]{AJO} and \cite[Proposition 2]{FGJ} for similar results in this setting for the usual wave front set, even for quasianalytic weight functions.
\begin{Prop}
\label{prop2FGJ}
Let $P(D)$ be a hypoelliptic linear partial differential operator of
order $m$ with constant coefficients, $\Omega$ an open subset of $\R^n$ and
$u\in\D'(\Omega)$.
Let $\sigma_0$ and $\omega$ be two
non-quasianalytic weight functions such that,
for $t\to+\infty$, $\sigma_0(t)=o(\omega(t))$ and $\omega(t^\gamma)=
o(t^{1/s})$ for $\gamma$ as in \eqref{gamma} and $s>1$.
Then
\beqsn
\WF_{\{\omega\}}^P(u)=\overline{\bigcup_{\sigma\in S}\WF_{(\sigma)}^P(u)},
\eeqsn
where $S:=\{\sigma$ {\em non-quasianalytic weight function: }
$\sigma_0\leq\sigma=o(\omega)\}$.
\end{Prop}

\begin{proof}

Let us first prove that
\beqs
\label{I1}
\overline{\bigcup_{\sigma\in S}\WF_{(\sigma)}^P(u)}\subseteq\WF_{\{\omega\}}^P(u).
\eeqs
To do this,  we fix $(x_0,\xi_0)\notin\WF_{\{\omega\}}^P(u)$.
By Corollary \ref{corWFR} there exist a neighborhood $U$ of $x_0$,
an open conic neighborhood $\Gamma$ of $\xi_0$ and
$\psi\in\D_{\{t^{1/s}\}}(\R^n)$ with $\psi\equiv1$ in $U$ such that \eqref{RR3}
is satisfied for some $k_0\in\N$ and $C>0$.

If $\sigma(t)=o(\omega(t))$, then for every $k\in\N$ there exists $t_k>0$
such that
\beqsn
k\sigma(t)\leq\frac{1}{k_0}\omega(t)
\qquad\forall t\geq t_k.
\eeqsn
Since $P$ is hypoelliptic there exists then $R_k>0$ such that
\beqsn
k\sigma(|P(\xi)|^{\frac1m})\leq\frac{1}{k_0}\omega(|P(\xi)|^{\frac1m})
\qquad\forall|\xi|\geq R_k,
\eeqsn
and therefore there exists $C_k>0$ such that, by \eqref{RR3},
\beqsn
|\widehat{\psi u}(\xi)|\leq C
e^{-\frac{1}{k_0}\omega(|P(\xi)|^{\frac1m})}
\leq C_k e^{-k\sigma(|P(\xi)|^{\frac1m})}
\qquad\forall\xi\in\R^n.
\eeqsn
This proves, by Corollary \ref{corWFB}, that $(x_0,\xi_0)\notin
\WF_{(\sigma)}^P(u)$ and hence
\beqsn
\bigcup_{\sigma\in S}\WF_{(\sigma)}^P(u)\subseteq
\WF_{\{\omega\}}^P(u).
\eeqsn

Since the wave front set $\WF_{\{\omega\}}^P(u)$ is always a closed set,
we have the inclusion \eqref{I1}.

Let us prove the other inclusion.
Take $(x_0,\xi_0)\notin\overline{\bigcup_{\sigma\in S}\WF_{(\sigma)}^P(u)}$.
Then there exist a compact neighborhood $K$ of $x_0$ and
a closed conic neighborhood $F$ of $\xi_0$ such that
\beqsn
(K\times F)\cap\overline{\bigcup_{\sigma\in S}\WF_{(\sigma)}^P(u)}
=\emptyset.
\eeqsn
Take $\chi_N\in\D(K)$ with $\chi_N\equiv1$ in a neighborhood $K'\subset K$
of $x_0$ which satisfies \eqref{B9}.
Take then $\psi\in\D_{\{t^{1/s}\}}(K')$. By Proposition \ref{lemma2'}
\beqsn
\WF_{(\sigma)}^P(\psi u)\subseteq\WF_{(\sigma)}^P(u)
\qquad\forall\sigma\in S
\eeqsn
and hence
\beqsn
(K\times F)\cap\WF_{(\sigma)}^P(\psi u)=\emptyset
\qquad\forall\sigma\in S.
\eeqsn
Consider then $g_N:=\chi_{Nmp}P(D)^N\psi u$, for $p$ sufficiently large
so that, by \cite[Cor. 10]{BJJ}:
\beqsn
&&\forall\sigma\in S,\forall k,\ell\in\N\ \exists C_{k,\ell,\sigma}>0\
\mbox{s.t.}\\
&&|\widehat{g}_N(\xi)|\leq  C_{k,\ell,\sigma}
e^{k\varphi^*_\sigma\left(\frac{Nm}{k}\right)}(1+|\xi|)^{-\ell}
\qquad\forall\xi\in F,
\eeqsn
where $\varphi^*_\sigma$ is the Young conjugate of
$\varphi_\sigma(t)=\sigma(e^t)$.

But
\beqsn
g_N=\chi_{Nmp}P(D)^N\psi u=P(D)^N\psi u
\qquad\mbox{in}\ \R^n
\eeqsn
since $\chi_{Nmp}\equiv1$ on $\supp\psi$.
Therefore
\beqsn
|P(\xi)^N\widehat{\psi u}(\xi)|=
|\widehat{g}_N(\xi)|\leq C_{k,\ell,\sigma}
e^{k\varphi^*_\sigma\left(\frac{Nm}{k}\right)}(1+|\xi|)^{-\ell}
\qquad\forall\xi\in F,
\eeqsn
which implies, by Lemma \ref{lemmaBR}:
\beqs
\label{I2}
|\widehat{\psi u}(\xi)|\leq C_{k,\sigma}e^{-k\sigma(|P(\xi)|^{1/m})}
\qquad\forall\xi\in F.
\eeqs

We want to prove that there exists $\bar{k}\in\N$ such that \eqref{RR3} is
satisfied on $F$ (this would imply, by Corollary \ref{corWFR}, that
$(x_0,\xi_0)\notin\WF_{\{\omega\}}^P(u)$).
We argue by contradiction and assume that for every $n\in\N$ there exists
$\xi_n\in F$ such that
\beqsn
|P(\xi_n)|\geq\delta|\xi_n|^d\to+\infty
\eeqsn
and
\beqs
\label{I3}
|\widehat{\psi u}(\xi_n)|\geq ne^{-\frac1n\omega(|P(\xi_n)|^{1/m})}
\qquad\forall n\in\N.
\eeqs

Since $\sigma(t)=o(\omega(t))$, for $\sigma\in S$,
for every $n\in\N$ there exists $k_n\geq n$ such that
\beqsn
\frac{\sigma(|P(\xi_k)|^{1/m})}{\omega(|P(\xi_k)|^{1/m})}<\frac 1n
\qquad\forall k\geq k_n.
\eeqsn
This would imply, together with \eqref{I3}, that
\beqsn
|\widehat{\psi u}(\xi_{k_n})|\geq k_n
e^{-\frac{1}{k_n}\omega(|P(\xi_{k_n})|^{1/m})}
\geq n
e^{-\frac{1}{n}\omega(|P(\xi_{k_n})|^{1/m})}
>n
e^{-\sigma(|P(\xi_{k_n})|^{1/m})},
\eeqsn
contradicting \eqref{I2} for $k=1$.

Therefore $(x_0,\xi_0)\notin\WF_{\{\omega\}}^P(u)$ and the
proposition is proved.
\end{proof}

\section{Operators with constant strength}
We give some applications to linear partial differential operators with variable coefficients with constant strength. We recall, from \cite{H}, the following:
\begin{Def}
\label{defES}
Let $P(D)$ and $Q(D)$ be two
linear partial differential operators with constant coefficients.
We say that $P$ is {\em weaker} than $Q$, and write $P\prec Q$, if
there exists a constant $C>0$ such that
\beqsn
\tilde{P}(\xi)\leq C\tilde{Q}(\xi)\qquad\forall\xi\in\R^n,
\eeqsn
where $\tilde{P}(\xi):=\sqrt{\sum_\alpha|D^\alpha P(\xi)|^2}$.
We say that $P$ and $Q$ are {\em equally strong} if there exists a
constant $C>0$ such that
\beqsn
C^{-1}\tilde{P}(\xi)\leq\tilde{Q}(\xi)\leq C\tilde{P}(\xi)
\qquad\forall \xi\in\R^n.
\eeqsn
\end{Def}

\begin{Rem}
\label{remES}
\begin{em}
If $P(D)$ and $Q(D)$ are equally strong and hypoelliptic, then by
Theorem \ref{thH} it follows that there are two constants $C,C'>0$
such that
\beqsn
&&|P(\xi)|^2\leq C(1+|Q(\xi)|^2)\qquad\forall\xi\in\R^n\\
&&|Q(\xi)|^2\leq C'(1+|P(\xi)|^2)\qquad\forall\xi\in\R^n.
\eeqsn
In particular, $\deg P=\deg Q$.
\end{em}
\end{Rem}

We recall from \cite[Thm. 11.1.9]{H} the following
\begin{Th}
\label{th1119horII}
If $P(D)$ and $Q(D)$ are equally strong and $P(D)$ is hypoelliptic, then
also $Q(D)$ is hypoelliptic. Moreover, if $d_P(\xi)$ and $d_Q(\xi)$
are the distance from $\xi\in\R^n$ to $V(P)$ and $V(Q)$ respectively,
 there exists then a constant $C>0$ such that
\beqs
\label{dp1}
C^{-1}\leq\frac{d_P(\xi)+1}{d_Q(\xi)+1}\leq C,\qquad
\forall\xi\in\R^n.
\eeqs
\end{Th}

\begin{Lemma}
\label{lemma3A}
Let $P(D)$ and $Q(D)$ be two equally strong linear partial differential operators. Assume that $P$ (and hence $Q$) is hypoelliptic and of order
$m$. Let $c_P$ and $c_Q$ be the constants defined in \eqref{defc}
for $P$ and $Q$ respectively. Then $c_P=c_Q$.
\end{Lemma}

\begin{proof}
Let us first remark that condition \eqref{dp1} is equivalent to
\beqs
\label{dp2}
\tilde{C}^{-1}\leq \frac{d_P(\xi)}{d_Q(\xi)}\leq \tilde{C},\qquad
|\xi|\gg1,
\eeqs
for some $\tilde{C}>0$, because of condition $(6)$ of Theorem \ref{thH}.

Then \eqref{defc} implies, for some $C,\bar{C}>0$,
\beqsn
&&|\xi|^{c_P}\leq C^{-1}d_P(\xi)\leq C^{-1}\tilde{C}d_Q(\xi)\qquad\Rightarrow
\quad c_Q\geq c_P\\
&&|\xi|^{c_Q}\leq \bar{C}^{-1}d_Q(\xi)\leq \bar{C}^{-1}
\tilde{C}d_P(\xi)\qquad\Rightarrow
\quad c_P\geq c_Q
\eeqsn
and $c_P=c_Q$.
\end{proof}

\begin{Rem}
\label{remgammap}
\begin{em}
If $P(D)$ and $Q(D)$ are equally strong, hypoelliptic and of order $m$, and
if $\gamma_P$ and $\gamma_Q$ are the constants defined in \eqref{gamma}
for $P$ and $Q$ respectively, then (see Remark \ref{rem22}):
\beqs
\label{gammap}
m\leq\gamma_P,\gamma_Q\leq\frac{m}{c_P}=\frac{m}{c_Q}\,.
\eeqs
\end{em}
\end{Rem}

\begin{Th}
\label{th1119BR}
Let $P(D)$ and $Q(D)$ be two equally strong linear partial differential operators
 and let $\omega$ be a non-quasianalytic weight function.
If $P$ is $*$-hypoelliptic then also $Q$ is $*$-hypoelliptic, for
$*=(\omega)$ or $\{\omega\}$.
\end{Th}

\begin{proof}
From Remark \ref{rem27} we have that $P$ is hypoelliptic.
Therefore also $Q$ is hypoelliptic, by Theorem \ref{th1119horII}, and
\eqref{dp2} is satisfied.

\underline{Beurling case}.
If $P$ is $(\omega)$-hypoelliptic, then by Theorem \ref{thhypoB}
\beqsn
\frac{\omega(\xi)}{d_Q(\xi)}\leq \tilde{C}
\frac{\omega(\xi)}{d_P(\xi)}\longrightarrow0\qquad\mbox{as}\
|\xi|\to+\infty
\eeqsn
and hence also $Q$ is $(\omega)$-hypoelliptic.

\underline{Roumieu case}.
If $P$ is $\{\omega\}$-hypoelliptic, then by Theorem \ref{thhypoR}
\beqsn
\omega(\xi)\leq C d_P(\xi)\leq C\tilde{C}d_Q(\xi),
\qquad|\xi|\gg1
\eeqsn
and hence also $Q$ is $\{\omega\}$-hypoelliptic.
\end{proof}

Now, from  Theorems \ref{thB} and \ref{thR}, and Corollaries \ref{corWFB} and
\ref{corWFR} it is easy to deduce the following:

\begin{Prop}
\label{propES}
Let $P(D)$ and $Q(D)$ be two equally strong  linear partial differential operators. Assume that $P$ (and hence $Q$) is hypoelliptic and of order
$m$.
Let $\omega$ be a non-quasianalytic weight function such that
$\omega(t^\gamma)=o(\sigma(t))$, as $t\to+\infty$, where
$\gamma:=\frac{m}{c_P}=\frac{m}{c_Q}$ as in
\eqref{gammap}
and $\sigma(t)=t^{1/s}$ for some
$s>1$. Let $\Omega$ be an open subset of $\R^n$ and $u\in\D'(\Omega)$.
Then
\beqs
\label{PES1}
\E_*^P(\Omega)=&&\E_*^Q(\Omega)\\
\label{PES2}
\WF_*^P(u)=&&\WF_*^Q(u)
\eeqs
for $*=(\omega)$ or $\{\omega\}$.
\end{Prop}

Let us now consider linear partial differential operators $P(x,D)$ with variable
coefficients on an open subset $\Omega$ of $\R^n$.
We recall from \cite{H} the following:
\begin{Def}
\label{defCS}
A linear partial differential operator $P(x,D)$, with $C^\infty$
coefficients on an open subset $\Omega$ of $\R^n$,
 is said to have {\em constant strength} in
$\Omega$ if, for every $x_0,y_0\in\Omega$, the differential operators
with constant coefficients $P(x_0,D)$ and $P(y_0,D)$ are equally strong.
\end{Def}

The following result is due to H\"ormander and Taylor (cf.
Theorems 13.4.1, 13.4.2 and 13.4.4 of \cite{H}):
\begin{Th}
\label{thHT}
Let $\Omega$ be an open subset of $\R^n$ and $P(x,D)$ a linear partial
differential operator with coefficients in $C^\infty(\Omega)$.
Assume that $P(x,D)$ has constant strength.
Then the following conditions are equivalent:
\begin{itemize}
\item[(1)]
$P(x,D)$ is hypoelliptic in $\Omega$, i.e.
\beqsn
\singsupp u=\singsupp P(\cdot,D)u, \qquad u\in\D'(\Omega);
\eeqsn
\item[(2)]
$P(x_0,D)$ is hypoelliptic for some $x_0\in\Omega$;
\item[(3)]
$P(x_0,D)$ is hypoelliptic for all $x_0\in\Omega$.
\end{itemize}
Moreover, if one of the above equivalent conditions is satisfied, then
$P(x,D)$ is micro-hypoelliptic in $\Omega$, i.e.
\beqsn
\WF(u)=\WF(P(\cdot,D)u),\qquad u\in\D'(\Omega).
\eeqsn
\end{Th}

In the case of $\omega$-hypoellipticity, from Theorem \ref{th1119BR} we
immediately obtain that for a linear partial differential operator with
$C^\infty$
coefficients $P(x,D)$ with constant strength, $P(x_0,D)$ is $*$-hypoelliptic
for some $x_0\in\Omega$ if and only if $P(x_0,D)$ is $*$-hypoelliptic for all
$x_0\in\Omega$, where $*=(\omega)$ or $\{\omega\}$.

As a consequence of Theorem \ref{thHT} and Lemma \ref{lemma3A} we have that
if $P(x,D)$ is a hypoelliptic linear partial differential operator
of order $m$ with $C^\infty$ coefficients and of constant strength
in an open subset $\Omega$ of $\R^n$,
there is then a unique constant $c_P\in(0,1]\cap\Q$ satisfying
\eqref{defc}, in the sense that
$c_{P(x_0,D)}=c_{P(x_1,D)}=:c_P$ for all $x_0,x_1\in\Omega$. We can then
uniquely define
\beqs
\label{gammap2}
\gamma_P:=\frac{m}{c_P}\,.
\eeqs

\begin{Cor}
\label{cor1}
Let $\Omega$ be an open subset of $\R^n$ and $u\in\D'(\Omega)$.
Let $P(x,D)$ be a linear partial differential operator with
coefficients in $C^\infty(\Omega)$.
Assume that $P(x,D)$ has constant strength in $\Omega$ and that
$P(x_0,D)$ is hypoelliptic for some $x_0\in\Omega$.
Let $\omega$ be a non-quasianalytic weight function such that
$\omega(t^\gamma)=o(\sigma(t))$, as $t\to+\infty$, where $\gamma=
\gamma_P$ is the
constant defined in \eqref{gammap2} and $\sigma(t)=t^{1/s}$ for some
$s>1$.
Then
\beqsn
\E_*^{P(x,D)}(\Omega)=&&\E_*^{P(x',D)}(\Omega)
\qquad\forall x,x'\in\Omega\\
\WF_*^{P(x,D)}(u)=&&\WF_*^{P(x',D)}(u)\qquad\forall x,x'\in\Omega
\eeqsn
for $*=(\omega)$ or $\{\omega\}$.
\end{Cor}

\begin{proof}
It follows from Theorem \ref{thHT} and Proposition \ref{propES}.
\end{proof}

\begin{Cor}
\label{cor2}
Let $\Omega$ be an open subset of $\R^n$ and $u\in\D'(\Omega)$.
Let $P(x,D)$ be a linear partial differential operator with
coefficients in $C^\infty(\Omega)$.
Assume that $P(x,D)$ has constant strength in $\Omega$ and that
$P(x_0,D)$ is hypoelliptic for some $x_0\in\Omega$.
Let $\omega$ be a non-quasianalytic weight function such that
$\omega(t^\gamma)=o(\sigma(t))$, as $t\to+\infty$, where $\gamma=
\gamma_P$ is the
constant defined in \eqref{gammap2} and $\sigma(t)=t^{1/s}$ for some
$s>1$.
Then
\beqsn
\WF_*(u)\subseteq\WF_*^{P(x_1,D)}(u)\cup\bigg(\bigcap_{x'\in\Omega}
\Sigma_{x'}\bigg)\qquad\forall x_1\in\Omega,
\eeqsn
where $*=(\omega)$ or $\{\omega\}$ and
\beqsn
\Sigma_{x'}:=\{(x,\xi)\in\Omega\times(\R^n\setminus\{0\}):\
P_{m(x')}(x',\xi)=0\}
\eeqsn
with $P_{m(x')}(x',\xi)$ the principal part of $P(x',\xi)$.
\end{Cor}

\begin{proof}
By Theorem 13 of \cite{BJJ} we have that
\beqsn
\WF_*(u)\subseteq&&\WF_*^{P(x',D)}(u)\cup\Sigma_{x'}
\qquad\forall x'\in\Omega\\
=&&\WF_*^{P(x_1,D)}(u)\cup\Sigma_{x'}\qquad\forall x_1,x'\in\Omega
\eeqsn
by Corollary \ref{cor1}. This proves the thesis.
\end{proof}

\begin{Cor}
\label{cor3}
Let $\Omega$ be an open subset of $\R^n$ and $u\in\D'(\Omega)$.
Let $P(x,D)$ be a linear partial differential operator with
coefficients in $C^\infty(\Omega)$.
Assume that $P(x,D)$ has constant strength in $\Omega$ and that
$P(x_0,D)$ is elliptic for some $x_0\in\Omega$.
Let $\omega$ be a non-quasianalytic weight function such that
$\omega(t^\gamma)=o(\sigma(t))$, as $t\to+\infty$, where $\gamma=
\gamma_P$ is the
constant defined in \eqref{gammap2} and $\sigma(t)=t^{1/s}$ for some
$s>1$.

Then $P(x_1,D)$ is $*$-micro-hypoelliptic for all $x_1\in\Omega$, i.e.
\beqsn
\WF_*(u)=\WF_*(P(x_1,D)u)\qquad\forall x_1\in\Omega,
\eeqsn
for $*=(\omega)$ or $\{\omega\}$.
\end{Cor}

\begin{proof}
By \cite[Rem. 14]{BJJ} we have that
\beqsn
\WF_*(u)=
\WF_*^{P(x_0,D)}(u)
\eeqsn
by the ellipticity of $P(x_0,D)$.

Therefore, by Corollary \ref{cor1} and \cite[Rem. 12 and Prop. 9]{BJJ},
for all $x_1\in\Omega$:
\beqsn
\WF_*(u)=&&\WF_*^{P(x_0,D)}(u)=\WF_*^{P(x_1,D)}(u)\\
=&&\WF_*^{P(x_1,D)}(P(x_1,D)u)\subseteq\WF_*(P(x_1,D)u)
\subseteq\WF_*(u),
\eeqsn
and hence the thesis.
\end{proof}

The following corollary can be proved directly with the results obtained in \cite{FGJ}, but we present it here as a consequence of the previous theorems.
\begin{Cor}
\label{cor4}
Let $\Omega$ be an open subset of $\R^n$ and $u\in\D'(\Omega)$.
Let $P(x,D)$ be a linear partial differential operator with
coefficients in $\E_{\{\sigma\}}(\Omega)$.
Assume that $P(x,D)$ has constant strength in $\Omega$ and that
$P(x_0,D)$ is elliptic for some $x_0\in\Omega$.
Let $\omega$ be a non-quasianalytic weight function such that
$\omega(t^\gamma)=o(t^{1/s})$, for $\gamma=\gamma_P$ as in
\eqref{gammap2} and $s>1$, and such that $\omega(t)=o(\sigma(t))$ for a
non-quasianalytic weight function $\sigma$ satisfying, for some $c>0$,
\beqs
\label{strongweight}
\int_1^{+\infty}\frac{\sigma(ty)}{t^2}dt\leq c\sigma(t)+c
\qquad y>0.
\eeqs

Then
\beqs
\label{star410}
\WF_*(u)=\WF_*(P(\cdot,D)u)=\WF_*(P(x_1,D)u)\qquad\forall x_1\in\Omega,
\eeqs
for $*=(\omega)$ or $\{\omega\}$.
\end{Cor}

\begin{proof}
By Theorem 3 of \cite{FGJ} we have that
\beqsn
\WF_*(u)=\WF_*(P(\cdot,D)u).
\eeqsn
By Corollary \ref{cor3} we have
\beqsn
\WF_*(u)=\WF_*(P(x_1,D)u)\qquad\forall x_1\in\Omega.
\eeqsn
Therefore also
\beqsn
\WF_*(P(\cdot,D)u)=\WF_*(P(x_1,D)u)\qquad\forall x_1\in\Omega
\eeqsn
is valid.
\end{proof}

\begin{Ex}
\label{exdavid}
\begin{em}
Let $q(\xi)\geq0$ be a hypoelliptic polynomial, $h,m,m'\in\N_0$,
$m>m'$, and consider the operator $P(x,D)$ with symbol
\beqsn
p(x,\xi):=|x|^{2h}q(\xi)^m+q(\xi)^{m'}.
\eeqsn

It was proved in \cite[Ex. 3.6]{FGJ1} that
there exist $0\leq\delta<\rho\leq1$
such that, for $a<d:=\rho-\delta<1$, $a/d<r<1$, $\sigma(t)=t^r$ and
$\omega(t)=o(t^a)$:
\beqs
\label{ith34}
|p(x,\xi)|\geq c e^{-\omega(\xi)}
\qquad\forall x\in\R^n,\ |\xi|\gg1
\eeqs
for some $c>0$, and
\beqs
\label{iith34}
\ |D_x^\alpha D_\xi^\beta p(x,\xi)|\leq C^{|\alpha|+|\beta|}\beta!
e^{\frac1k\varphi^*_\sigma(|\alpha|k)}
|p(x,\xi)|(1+|\xi|)^{-\rho|\beta|+\delta|\alpha|}\qquad
\forall x\in\R^n, |\xi|\gg1
\eeqs
for some $C>0$ and $k\in\N$.

By \cite[Thm. 3.4]{FGJ1} and \cite[Thm. 2]{FGJ} we have
\beqsn
\WF_{(\omega)}(u)=\WF_{(\omega)}(P(\cdot,D)u),
\qquad u\in\E'_{(\omega)}(\Omega).
\eeqsn

Analogously, since \eqref{ith34} and \eqref{iith34} hold also for every fixed
$x_1\in\Omega$, we have that
\beqsn
\WF_{(\omega)}(u)=\WF_{(\omega)}(P(x_1,D)u),
\qquad
\forall x_1\in\Omega,
\eeqsn and hence
\eqref{star410} holds for all
$u\in\E'_{(\omega)}(\Omega)$, and hence for all $u\in\E'(\Omega)$.

Note that $P$ is not of constant strength, and that $P$ is elliptic in
some $x_0\in\Omega$ if and only if $q(\xi)$ is elliptic.
\end{em}
\end{Ex}

\begin{Ex}
\label{exrodino}
\begin{em}
Let us consider, for $m\geq1$, the operator
\beqsn
P(x,y,D_x,D_y)=D_x^{2m}+x^{2h}D_y^{2m+2}.
\eeqsn
It was proved in \cite[Thm. 3.1]{R2} that if $h>m$
and $\omega(t)=t^{1/s}$, with $s\geq1+\frac1m$, then
\beqs
\label{E1}
\WF_{\{\omega\}}(u)=\WF_{\{\omega\}}(P(\cdot,D)u),
\qquad\forall u\in\D'(\R^2).
\eeqs

However, for fixed $(x_1,y_1)=(0,y_1)$ the operator
$P(0,y_1,D_x,D_y)=D_x^{2m}$ is not $\{\omega\}$-hypo\-el\-lip\-tic, since
every $f(x,y)=f(y)\in\D'(\R)\setminus\E_{\{\omega\}}(\R)$ solves
$P(0,y_1,D_x,D_y)f=0$, and hence it is not $\{\omega\}$-microhypoelliptic
by \cite[Cor. 11]{BJJ}:
\beqsn
\WF_{\{\omega\}}(P(0,y_1,D_x,D_y)f)\not\subseteq\WF_{\{\omega\}}(f).
\eeqsn
In particular, by \eqref{E1}:
\beqsn
\WF_{\{\omega\}}(P(0,y_1,D_x,D_y)f)\not\subseteq\WF_{\{\omega\}}(P(\cdot,D)f).
\eeqsn

Note that $P$ does not have constant strength and $P(x_0,y_0,D_x,D_y)$ is
not elliptic for any $(x_0,y_0)\in\R^2$.
\end{em}
\end{Ex}

\vspace{3mm}
{\bf Acknowledgements:} 
The authors were partially supported by FAR2011 (Universit\`a di Ferrara) and "Fondi per le necessit\`a di base della ricerca" 2012 and 2013 (Universit\`a di Ferrara). 
The first author is member of the
Gruppo Nazionale per l'Analisi Matematica, la Probabilit\`a e le loro
Applicazioni (GNAMPA) of the Instituto Nazionale di Alta Matematica (INdAM).
The research of the second author was partially supported by MINECO of Spain, Project MTM2013-43540-P and by Programa de Apoyo a la Investigaci\'on y Desarrollo de la UPV, PAID-06-12.


\begin{thebibliography}{999}

\bibitem{AJO}
A.A.~Albanese, D.~Jornet, A.~Oliaro, {\em Quasianalytic wave front sets
for solutions of linear partial differential operators}, Integr. Equ.
Oper. Theory { 66}  (2010), 153-181.

\bibitem{BJ}
C.~Boiti, D.~Jornet, {\em The problem of iterates in some classes of
ultradifferentiable functions}, to appear in ``Operator Theory:
Advances and Applications'', Birkhauser, Basel.

\bibitem{BJJ}
C.~Boiti, D.~Jornet, J.~Juan-Huguet, {\em Wave front set with
respect to the iterates of an operator with constant coefficients},
Abstr. Appl. Anal. { 2014}, Article ID 438716 (2014), pp. 1-17,
http://dx.doi.org/10.1155/2014/438716

\bibitem{BCM} P.~Bolley, J.~Camus,
C.~Mattera, {\em Analyticit\'e microlocale et it\'er\'es d'operateurs
hypoelliptiques}, S\'eminaire Goulaouic-Schwartz, 1978-79, Exp N.13,
\'Ecole Polytech., Palaiseau.

\bibitem{BFM} J.~Bonet, C.~Fern\'andez, R.~Meise, {\em Characterization
of the $\omega$-hypoelliptic convolution operators on
ultradistributions}, Ann. Acad. Sci. Fenn. Math. { 25} (2000), 261-284


\bibitem{bonet_meise_melikhov2007a} J. Bonet, R. Meise, S.N. Melikhov,
{\em A comparison of two different ways of define classes of
ultradifferentiable functions}, Bull. Belg. Math. Soc. Simon Stevin 14 (2007),
425-444.

\bibitem{BMT} R.W.~Braun, R.~Meise, B.A.~Taylor,
{\em Ultradifferentiable functions and Fourier analysis}, Result. Math.
{ 17} (1990), 206--237.

\bibitem{FGJ1} C. Fern\'andez, A. Galbis, D. Jornet,
{\em $\omega$-hypoelliptic differential operators of
constant strength}, J. Math. Anal. Appl. { 297} (2004), 561-576.

 \bibitem{FGJ} C. Fern\'andez, A. Galbis, D. Jornet,
{\em Pseudodifferential operators of Beurling type and the wave
front set}, J. Math. Anal. Appl. {340} (2008), 1153-1170.

\bibitem{Hint} L.~H\"ormander,
{\em On Interior Regularity of the Solutions of Partial
Differential Equations}, Comm. Pure Appl. Math. Vol. XI (1958), 197-218.

\bibitem{Huniq} L.~H\"ormander, {\em Uniqueness theorems
and wave front sets for solutions of linear partial differential equations
with analytic coefficients}, Comm. Pure Appl. Math. {24} (1971), 671--704.

\bibitem{H1} L. H\"ormander, {\em The Analysis of Linear Partial
Differential Operators I}, Springer-Verlag, Berlin (1990).

\bibitem{H} L. H\"ormander, {\em The Analysis of Linear Partial
Differential Operators II}, Springer-Verlag, Berlin (1983).

\bibitem{juan-huguet2010iterates} J.~Juan-Huguet,~{\em Iterates and
Hypoellipticity of Partial Differential Operators on Non-Quasianalytic
Classes}, Integr. Equ. Oper. Theory {68} (2010), 263-286.


\bibitem{J} J.~Juan-Huguet, {\em A Paley-Wiener type theorem for
generalized non-quasianalytic classes}, Studia Math. { 208}, n.1 (2012),
31-46.


\bibitem{K1} H. Komatsu, {\em A characterization of
real analytic functions}, Proc. Japan Acad. {36} (1960), 90-93.


\bibitem{KN} T. Kotake, M.S. Narasimhan,
{\em Regularity theorems for fractional powers of
a linear elliptic operator}, Bull. Soc. Math. France {90} (1962), 449-471.

\bibitem{langenbruch1979P} M. Langenbruch,
{\em P-Funktionale und Randwerte zu hypoelliptischen Differentialoperatoren},
Math. Ann. 239(1) (1979), 55-74.


\bibitem{langenbruch1979Fortsetzung} M. Langenbruch,
{\em Fortsetzung von Randwerten zu hypoelliptischen Differentialoperatoren
und partielle Differentialgleichungen}, J. Reine Angew.
Math. 311/312 (1979), 57-79.


\bibitem{langenbruch1985on} M. Langenbruch,
{\em On the functional dimension of solution spaces of hypoelliptic
partial differential operators}, Math. Ann. 272 (1985), 217-229.


\bibitem{langenbruch1987bases} M. Langenbruch,
{\em Bases in solution sheaves of systems of partial differential
equations}, J. Reine Angew. Math. 373  (1987), 1-36.

\bibitem{metivier1978propiete} G. M\'etivier, {\em Propri\'et\'e des
it\'er\'es et ellipticit\'e}, Comm. Partial Differential Equations 3 {(9)}
(1978), 827-876.

\bibitem{NZ} E. Newberger, Z. Zielezny, {\em The growth
of hypoelliptic polynomials and Gevrey classes}, Proc. Amer. Math. Soc.
{39}, n. 3 (1973), 547-552.

\bibitem{R2} L.~Rodino, {\em On the problem of the hypoellipticity
of the linear partial differential equations}, Developments in Partial
Differential Equations and Applications to Mathematical
Physics, Edited by G.~Buttazzo, Plenum Press, New York, 1992.

\bibitem{R} L.~Rodino, {\em Linear partial differential operators in
Gevrey spaces}, World Scientific, 1993.

\bibitem{Z1} L. Zanghirati, {\em Iterates of a class of hypoelliptic
operators and generalized Gevrey classes}, Boll. U.M.I. Suppl. 1 (1980),
177-195.


\end{thebibliography}
\end{document}